\documentclass{amsart}
\usepackage{amssymb,latexsym}
\usepackage{amsfonts}
\usepackage[mathscr]{eucal}
\usepackage[all]{xy}
\usepackage{enumerate}
\usepackage[colorlinks,linkcolor=black,anchorcolor=black,citecolor=black]{hyperref}
\newtheorem{thm}{Theorem}[section]
\newtheorem{yl}[thm]{Lemma}
\newtheorem{tl}[thm]{Corollary}
\newtheorem{mt}[thm]{Proposition}
\theoremstyle{remark}
\newtheorem{zj}{Remark}[section]
\theoremstyle{definition}
\newtheorem{dy}{Definition}[section]
\numberwithin{equation}{section}
\allowdisplaybreaks[4]

\newcommand{\dd}{\operatorname{d}}

\newcommand{\Aut}{\operatorname{Aut}}

\newcommand{\IM}{\operatorname{Im}}

\newcommand{\im}{\sqrt{-1}}

\newcommand{\lr}[1]{(#1)}

\newcommand{\ra}{\rightarrow}

\newcommand{\laplace}{\bigtriangleup}
\newcommand{\grad}{\nabla}

\newenvironment{eqn}{\begin{equation}}{\end{equation}}
\newcommand\relphantom[1]{\mathrel{\phantom{#1}}}
\newcommand\repl{\relphantom}

\begin{document}
\title{Twisted and conical K\"{a}hler-Ricci soliton on Fano manifold}
\author{Xishen Jin}
\address{Xishen Jin\\Key Laboratory of Wu Wen-Tsun Mathematics\\ Chinese Academy of Sciences\\School of Mathematical Sciences\\
University of Science and Technology of China\\
Hefei, 230026, P.R. China\\} \email{jinxsh@mail.ustc.edu.cn}
\author{JiaWei Liu}
\address{Jiawei Liu\\Key Laboratory of Wu Wen-Tsun Mathematics\\ Chinese Academy of Sciences\\School of Mathematical Sciences\\
University of Science and Technology of China\\
Hefei, 230026, P.R. China\\} \email{liujw24@mail.ustc.edu.cn}
\author{Xi Zhang}
\address{Xi Zhang\\Key Laboratory of Wu Wen-Tsun Mathematics\\ Chinese Academy of Sciences\\School of Mathematical Sciences\\
University of Science and Technology of China\\
Hefei, 230026,P.R. China\\ } \email{mathzx@ustc.edu.cn}
\thanks{The authors were supported in part by NSF in
China No.11131007 and the Hundred Talents Program of CAS}
\begin{abstract}
In this paper, we consider the twisted K\"ahler-Ricci soliton, and show that the existence of twisted K\"ahler-Ricci soliton with semi-positive twisting form is closely related to the properness of some energy functionals.  We also consider the conical K\"ahler-Ricci soliton, and obtain some existence results. In particular, under some assumptions on the divisor and $\alpha$-invariant, we get the properness of the modified log K-energy and the existence of conical K\"ahler-Ricci soliton with suitable cone angle.
\end{abstract}
\maketitle
\section{Introduction}
\label{section:1}
Let $\lr{M,J}$ be a compact Fano manifold.  A K\"ahler metric $\omega \in 2\pi c_{1}(M)$ is called a K\"{a}hler--Ricci soliton if there exits a holomorphic vector field $X$ over $M$ such that $$Ric\lr{\omega}=\omega+L_X\omega .$$
  K\"{a}hler--Ricci solitons  can be considered as a natural extension of the K\"{a}hler--Einstein metrics, which have been studied by Cao \cite{cao2012existence}, Cao-Tian-Zhu \cite{cao2005kahler}, Hamilton \cite{hamilton1993eternal}, Tian \cite{tian1997kahler}, Tian-Zhu \cite{tian2002new}, Zhu \cite{zhu2000kahler}, etc. Specially, in \cite{cao2005kahler}, the authors shown that the existence of K\"{a}hler--Ricci solitons is
 closely related to the properness of the modified Ding--functional or Mabuchi $K$--energy. Following the works of Aubin \cite{aubin1976equations} and Yau \cite{yau1978ricci}, we study the continuity method used in \cite{cao2005kahler}. Given a K\"{a}hler metric $\omega_0\in 2\pi c_1\lr{M}$ and holomorphic vector field $X$, the approach is to find $\omega_{t}$ solving the following equation,
\begin{eqn}
\label{eqn:1}
Ric\lr{\omega_{t}}=t\omega_{t}+\lr{1-t}\omega_{0}+L_X\omega_{t}
\end{eqn}
for all $t\in [0,1]$. In this paper, we are interested in the supremum of $t$ for which we can solve the equation above.

Let $\Aut\lr{M}$ be the connected component containing the identity holomorphism transformation and $\eta\lr{M}$ be its Lie algebra consisting of all holomorphic vector fields on $M$. According to \cite{futaki1995bilinear}, there exists a semidirect decomposition of $\Aut\lr{M}$, such that $$\Aut\lr{M}=\dot{\Aut}\lr{M}\propto R_u,$$ where $R_u$ is the unipotent radical of $\Aut\lr{M}$ and $\dot{\Aut}\lr{M}\subset \Aut\lr{M}$ is the reductive subgroup as a complexification of a subgroup $K$ of $\Aut\lr{M}$, where $K$ is the maximal compact subgroup of $\Aut\lr{M}$ containing the one--parameter transformations subgroup of $\Aut\lr{M}$ generating by $\operatorname{Im}X$. Obviously, the Lie subalgebra $\dot{\eta}\lr{M}$ of $\dot{\Aut}\lr{M}$ is reductive. More precisely, $\dot\eta\lr{M}$ is the complexification of a real compact Lie algebra of $K$. In particular, $X\in \dot{\eta}\lr{M}$. Now we assume $\omega_0 \in 2\pi c_1\lr{M}$ is a smooth K\"{a}hler metric invariant under the action of $\Phi_{\operatorname{Im}X}$, where $\Phi_{\operatorname{Im}X}$ is the one--parameter transformations subgroup of $\Aut\lr{M}$ generating by $\operatorname{Im}X$, i.e. $$L_{\operatorname{Im}X}\omega_0=0.$$

Let $\mathscr{H}\lr{M,\omega_0}=\{\varphi\in L_{loc}^1(M)|\text{ }\omega_0+\im\partial\overline{\partial}\varphi>0 \text{ in the sense of current}\}$. As in \cite{cao2005kahler}, we define the following function subspace of $\mathscr{H}\lr{M,\omega_0}$:$$\mathscr{H}_X\lr{M,\omega_0}=\{\varphi\in \mathscr{H}\lr{M,\omega_0}\cap C^\infty\lr{M}|\text{ }\IM\lr{X}\varphi=0\}.$$ We define $\mathscr{K}_X^{0}(\omega_0)$ to be the space of smooth semipositive $(1,1)$--forms cohomology to $\omega_0$, i.e.
\begin{equation*}
  \mathscr{K}_X^{0}(\omega_0)=\{\omega\in[\omega_0]|\text{ }\omega \text{ is smooth and }\omega\geq 0 \text{, }L_{\operatorname{Im}X}\omega=0\}.
\end{equation*}
Furthermore, $\mathscr{K}_X\lr{\omega_0}$ is a subspace of K\"ahler metrics defined as follow:
\begin{equation*}
  \mathscr{K}_X(\omega_0)=\{\omega\in [\omega_0]|\text{ }\omega \text{ is a K\"ahler metric and }L_{\operatorname{Im}X}(\omega)=0\}.
\end{equation*}

\begin{dy}
  We define the following invariant with respect to $X$,
  \begin{eqn}
    R\lr{X}=\sup\{\beta|\ \exists\text{ }\omega\in \mathscr{K}_X\lr{\omega_0},\text{ such that } Ric\lr{\omega}-L_X\omega\geq \beta\omega\}.
  \end{eqn}
\end{dy}

\begin{zj}
 According to \cite{zhu2000kahler}, there exists $\omega'_{0}\in \mathscr{K}_X(\omega_0)$, such that
\begin{equation*}
  Ric(\omega'_0)-L_X\omega'_0=\omega_0\geq c'\omega'_0 >0,
\end{equation*}
so we get $R(X)>0$. Furthermore, by taking integration over $M$ on both sides of $Ric\lr{\omega}-L_X\omega\geq \beta\omega$, we conclude that $\beta\leq 1$, therefore $0< R(X) \leq 1$. If $X\equiv 0$, $R(0)$ is just the invariant defined in \cite{GS}.
\end{zj}

Note that $\omega$ is a closed form and $X$ is holomorphic, we have that $\overline{\partial}(i_X\omega)=0$. According to the Hodge decomposition theorem and the property of Fano manifold, we can find a smooth real-valued function $\theta_X(\omega)$ such that for $\omega \in \mathscr{K}_X(\omega_0)$,
\begin{equation}
\label{eqn:2}
  i_X\omega=\im \overline{\partial} \theta_X(\omega)
\end{equation}
and $\theta_X(\omega)$ satisfies the normalization
\begin{equation*}
  \int_Me^{\theta_X(\omega)}\omega^n=\int_M \omega_0^n.
\end{equation*}
We will take notation that $\theta_X=\theta_X(\omega_0)$ in the whole paper without special instruction. By direct computation, we get that $\theta_X(\omega_\varphi)=\theta_X+X(\varphi)$.

\begin{dy}
  For any $\lr{1,1}$--form $\eta\in\lr{1-\beta}\mathscr{K}^0_X\lr{\omega_0},$ we say a K\"{a}hler metric $\omega\in \mathscr{K}_X\lr{\omega_0}$ is a twisted K\"{a}hler--Ricci soliton with respect to $\eta$ if it satisfies
  \begin{equation}
  \label{eqn:GKS}
    Ric\lr{\omega}=\beta\omega+\eta+L_X\omega.
  \end{equation}
\end{dy}

\begin{zj}
It is easy to see that finding the twisted K\"{a}hler--Ricci soliton as \eqref{eqn:GKS} is equivalent to solving the following Monge-Amp\`{e}re equation:
\begin{equation}
\label{eqn:MAKRS}
  \frac{\omega_\varphi^n}{\omega_0^n}=e^{h_{\omega_0}-\beta\varphi-\theta_X-X\lr{\varphi}},
\end{equation}
where $h_{\omega}$ is the Ricci potential defined by
\begin{equation}
\label{Ricci pote}
  Ric(\omega)-\beta\omega-\eta=\im\partial\overline{\partial}h_\omega
\end{equation}
 normalized such that $\int_Me^{h_\omega}\omega^n =\int_M\omega^n$. And we just consider the case when $\beta$ is nonnegative, since the equation \eqref{eqn:MAKRS} is solvable according to the celebrated work of Aubin \cite{aubin1976equations} and Yau \cite{yau1978ricci} on the other case.
\end{zj}

In this paper, we will follow Tian's argument in \cite{tian1997kahler}  to show that the existence of twisted K\"{a}hler--Ricci soliton with respect to $\eta$ is closely related to the properness of the twisted $K$--energy $\widetilde{\mu}_{\omega_0,\eta}$ (see the definition in section \ref{section:1}). Following the discussion of Tian--Zhu \cite{tian2000nonlinear} and Phong--Song--Strum--Weinkove \cite{phong2008moser}, we deduce a linear Moser--Trudinger type inequality. In fact, we get the first main theorem.
\begin{thm}
\label{theorem:1}
  Let $\lr{M,\omega_0}$ be a compact K\"{a}hler manifold, $L_{\IM X}\omega_0=0$ and $\eta$ is a real closed semipositive $\lr{1,1}$--form in $\lr{1-\beta}\mathscr{K}^0_X\lr{\omega_0}$ with $\beta>0$ , where $X$ is a holomorphic vector field on $M$. Suppose the twisted $K$--energy $\widetilde{\mu}_{\omega_0,\eta}$ is proper. Then there is a twisted K\"{a}hler--Ricci soliton $\omega\in [\omega_0]$ with respect to $\eta$, i.e.
   \begin{eqn}\label{1.7}
     Ric\lr{\omega}=\beta\omega+\eta+L_X\omega.
   \end{eqn}
   On the other hand, assuming that the twisted form $\eta$ is strictly positive at a point, if there exists a twisted K\"{a}hler--Ricci soliton $\omega_{TKS}\in \mathscr{K}_X \lr{\omega_0}$ with respect to $\eta$, then $\widetilde{\mu}_{\omega_0, \eta}$ must be proper. Furthermore, there exist two positive constants $C_1$ and $C_2$ depending only on $\beta$, $\eta$, $X$ and the geometry of $\lr{M,\omega_{TKS}}$, such that
  \begin{eqn}\label{1.8}
    \widetilde{\mu}_{\omega_{0},\eta}\lr{\varphi}\geq C_1 \widetilde{J}_{\omega_{0}}\lr{\varphi}-C_2
  \end{eqn}
  for all $\varphi\in \mathscr{H}_X\lr{M,\omega_{0}}$.
\end{thm}

\begin{zj}
$\repl{=}$

\begin{enumerate}
  \item A functional $G$ is $J$--proper on the function space $\mathscr{H}$ if there exists an increasing function \(  f:\mathbb{R}\rightarrow\left[c,+\infty\right)\) satisfying $\mathop{\lim}\limits_{t\rightarrow +\infty}f\lr{t}=+\infty$, such that
      \begin{equation*}
         G\lr{\phi}\geq f\lr{J\lr{\phi}}
      \end{equation*}
      for any $\phi\in \mathscr{H}$.
  \item When $\eta$ is strictly positive at one point, following the argument in \cite{berndtsson2013brunn}, we can get the uniqueness of twisted K\"ahler-Ricci soliton.
\end{enumerate}
\end{zj}

\medskip

If $\beta \in (0, R(X))$, by the definition of $R(X)$, there exists a K\"ahler metric $\tilde{\omega } \in \mathscr{K}_X\lr{\omega_0}$ such that
$$Ric\lr{\tilde{\omega}}-\beta\tilde{\omega}- L_X\tilde{\omega}>0.$$
Let $\eta = Ric\lr{\tilde{\omega}}-\beta\tilde{\omega}- L_X\tilde{\omega} \in (1-\beta)\mathscr{K}_X\lr{\omega_0} $. The equation (\ref{1.7}) can be solved in $\mathscr{K}_X\lr{\omega_0}$. By Theorem 1.1, we obtain that the twisted $K$--energy $\widetilde{\mu}_{\omega_0,\eta}$ is proper, in fact, it satisfies the Moser-Trudinger type inequality (\ref{1.8}). On other hand, by the definition of $\widetilde{\mu}_{\omega_0,\eta}$,  it is easy to see that the properness of the twisted $K$--energy $\widetilde{\mu}_{\omega_0,\eta}$ is independent on the choice of the twisting form $\eta \in (1-\beta)\mathscr{K}_X\lr{\omega_0}$, which implies that $\widetilde{\mu}_{ \omega_0 , (1-\beta ) \omega_{0}}$ is also proper. Then the equation (\ref{eqn:1}) can be solved at $t=\beta$. Furthermore, since $\omega_0$ is strictly positive, we have the following corollary,
%
%
%
\begin{tl}
\label{corallary:1}
  Let $\lr{M,\omega_0}$ be a K\"{a}hler manifold with $\omega_0\in 2\pi c_1\lr{M}$, and $0<\beta<1$. The following conditions are equivalent:
  \begin{enumerate}
    \item we can solve the equation \eqref{eqn:1},
    \item there exists a K\"{a}hler metric $\omega\in \mathscr{K}_X\lr{\omega_0}$ such that $Ric\lr{\omega}-L_X\lr{\omega}>\beta\omega$,
    \item for any K\"{a}hler metric $\omega\in \mathscr{K}_X\lr{\omega_0}$, $\widetilde{\mu}_{\omega}+\lr{1-\beta}\lr{\widetilde{I}_{\omega} - \widetilde{J}_{\omega}}$ is proper.
  \end{enumerate}
\end{tl}
%

Let $D=\{s=0\}\in|L|$ be a smooth divisor, and particularly, in this paper we consider $L$ is a holomorphic line bundle such that $c_1(L)=\lambda c(M)$, for some $\lambda\in \mathbb{Q}^+$. A smooth conical K\"{a}hler metric on $M$ with angle $2\pi\beta(0<\beta<1)$ along $D$ is a closed positive $\lr{1,1}$--current which is a smooth K\"{a}hler metric in $M\setminus D$ and asymptotically equivalent to the model conical metric
 \begin{equation*}
   \im\sum_{j=1}^{n-1}\dd z^j\wedge\dd \overline{z}^j+\im|z^n|^{2\beta-2}\dd z^n\wedge\dd \overline{z}^n,
 \end{equation*}
where $\lr{z^1,\cdots,z^n}$ are local holomorphic coordinates such that $D=\{z^n=0\}$. As in \cite{datar2013connecting}, we give the definition of conical K\"{a}hler--Ricci soliton with respect to the holomorphic vector field $X$, which has been studied on the toric manifold by \cite{datar2013connecting} and \cite{wang2014toric}.
\begin{dy}
  A conical K\"{a}hler metric $\omega\in 2\pi c_1\lr{M}$ is called a conical K\"{a}hler--Ricci soliton with respect to $X$ if $\omega$ satisfies:
  \begin{enumerate}
    \item The metric potential of $\omega$ is H\"{o}lder continuous with respect to $\omega_0$ on $M$,
    \item $Ric\lr{\omega}=\gamma(\lambda,\nu)\omega+\nu[D]+L_X\omega$ globally on $M$ in the sense of current, where $[D]$ is the current of integration along $D$,
    \item $Ric\lr{\omega}=\gamma(\lambda,\nu)\omega+L_X\omega$ on $M\backslash D$ in the classical sense,
  \end{enumerate}
  where $\gamma(\lambda,\nu)=1-\lambda\nu$.
\end{dy}

\begin{zj}
  While dealing with $L_X\omega$ as a current, we mean that for any smooth $\lr{n-1,n-1}$--form $\zeta$,
  \begin{equation*}
    \int_M L_X\omega\wedge\zeta=-\int_M \omega\wedge L_X\zeta.
  \end{equation*}
\end{zj}

\medskip

We will show that the existence of conical K\"ahler-Ricci soliton is also closely related to the properness of the log modified Mabuchi $K$--energy  $\widetilde{\mu}_{\omega_0,\nu D}$ and the log Ding functional $\widetilde{F}_{\omega_0, \nu D}$ which will be defined in section \ref{section:5}.
\begin{thm}
\label{theorem:existence}
   Assume that $X(\log|s|_H^2)$ is bounded, where $s$ is the defined section of $D$ and $H$ is a Hermitian metric on the line bundle $L$. If $\widetilde{\mu}_{\omega_0,\nu D}$ or $\widetilde{F}_{\omega_0, \nu D}$ is proper on the function space $\mathscr{H}_X\lr{M,\omega_0}$, where $0<\nu <1$, then there exists a conical K\"{a}hler--Ricci soliton $\omega_\nu$ with angle $2 \pi ( 1 - \nu ) $ along $D$, i.e. $\omega_\nu$ satisfies:
  \begin{equation}
 \label{eqn:D2}
    Ric\lr{\omega_\nu}=\gamma\lr{\lambda,\nu}\omega_\nu + \nu [D] +L_X\omega_\nu.
  \end{equation}
\end{thm}

\medskip

%
%

\medskip

Under some assumptions on the divisor and $\alpha$-invariant, modifying Berman's work \cite{berman2013thermodynamical}, we obtain the existence of conical K\"ahler-Ricci soliton for suitable cone angles, i.e. Theorem \ref{berman}. when $R\lr{X}=1$, we prove that the supremum of the cone angle  of conical K\"ahler-Ricci soliton must be $2\pi $, i.e. we get the following theorem.

\medskip


\begin{thm}
\label{theorem:0.4}
  Assume that $R(X)=1$, $D\in|L|$, $|X(\log|s|_H^2)|<C<+\infty$ and
  \begin{enumerate}
    \item $\tilde{C}<\lambda$,
    \item $\min\{\alpha(\omega_0),\lambda\alpha(L_{|D})\}> \max\{\frac{\tilde{C}(1-\lambda)}{(1-\tilde{C})},0\}$,
  \end{enumerate}
  where $\tilde{C}$ is the positive constant $C_2 <1$ in Proposition \ref{prop:1}, and $\alpha(\omega_0)$ and $\alpha(L_{|D})$ are the alpha invariants defined by Tian. For any $\beta\in\lr{\max\{\frac{1-\lambda}{1-\tilde{C}},0\},1}$, there exists a conical metric $\omega_\beta$ with the cone angle $2\pi(1- \frac{1-\beta}{\lambda}) $ such that
  \begin{equation*}
    Ric\lr{\omega_\beta}=\beta\omega_\beta+\frac{1-\beta}{\lambda}[D] + L_X\omega_\beta.
  \end{equation*}
  Furthermore, $\omega_\beta$ is a Gromov--Hausdorff limit of smooth twisted K\"{a}hler--Ricci solitons.
\end{thm}

\begin{zj}
  Our argument in Theorem \ref{theorem:0.4} can also be applied to the conical K\"ahler-Einstein case, i.e. $X\equiv 0$. When $L$ is just $-\lambda K_M$ for $\lambda\geq 1$, according to Bermann \cite{berman2013thermodynamical} and Li-Sun \cite{li2012conical}, the log Mabuchi K-energy for small cone angle is proper. If $R(0)$ the greatest lower bound of Ricci tensor defined in \cite{song2012greatest} is $1$, then we can get that the log Mabuchi K-energy is proper for any cone angle in $(0,2\pi)$. This result has been proved by Chi Li in \cite{li2013yau}, but our argument is different to that in \cite{li2013yau}.
\end{zj}

We will organize the paper as follow. In section \ref{section:2}, we will introduce the twisted K\"ahler--Ricci soliton and related functionals, i.e. twisted Mabuchi $K$--energy and Ding--functional. In section \ref{section:3}, we will discuss the existence of twisted K\"ahler--Ricci solitons. Then, we give a necessary condition for the existence of twisted K\"ahler--Ricci solitons, i.e. a version of Moser--Trudinger inequality, which will be proved in section \ref{section:4}. In section \ref{section:5}, we will prove the existence of conical K\"ahler--Ricci solitons under properness assumption of the log modified Mabuchi $K$--energy or log modified Ding--functional. In the last section, we find some condition under which we can get the properness, and we consider the limit behavior of a sequence of twisted K\"ahler--Ricci solitons in the sense of Gromov--Hausdorff distance.

\section{Some twisted functionals}
\label{section:2}
In this section, $\eta$ will be a fix $\lr{1,1}$--form in $\lr{1-\beta}\mathscr{K}^0_X\lr{\omega_0}\lr{0<\beta < 1}$.
Firstly, let us recall the modified Aubin--Yau functional $\widetilde{J}_{\omega_0}$ and $\widetilde{I}_{\omega_0}$ defined on $\mathscr{H}_X\lr{M,\omega_0}$ in \cite{tian2000uniqueness}:
\begin{equation*}
  \widetilde{J}_{\omega_0}\lr{\varphi}=\frac{1}{V}\int_0^1\int_M\dot{\varphi}_t \lr{e^{\theta_X}\omega_0^n - e^{\theta_X\lr{\omega_{\varphi_s}}}\omega_{\varphi_s}^n}\wedge\dd t,
\end{equation*}
\begin{equation*}
  \widetilde{I}_{\omega_0}\lr{\varphi}=\frac{1}{V}\int_M\varphi\lr{e^{\theta_X}\omega_0^n - e^{\theta_X\lr{\omega_{\varphi}}}\omega_{\varphi}^n},
\end{equation*}
where $\{\varphi_t\}\lr{0\leq t\leq 1}$ is a smooth path in $\mathscr{H}_X\lr{M,\omega_0}$ connecting $0$ and $\varphi$, and $\theta_X\lr{\omega}$ is defined as in \eqref{eqn:2}. If $X=0$, then $\widetilde{J}_{\omega_0}$ and $\widetilde{I}_{\omega_0}$ are just the original Aubin-Yau function $I_{\omega_0}$ and $J_{\omega_0}$.

\begin{mt}[\cite{cao2005kahler}]
\label{prop:1}
  $I_{\omega_0}$, $J_{\omega_0}$, $\widetilde{I}_{\omega_0}$ and $\widetilde{J}_{\omega_0}$ are positive on $\mathscr{H}_X\lr{M,\omega_0}$. There exist positive constants $C_1$ , $C_2$ , $C_3$ and $C_4$, where $C_1$ and $C_2$ is less than $1$ such that for any $\varphi\in \mathscr{H}_X\lr{M,\omega_0}$, such that
  \begin{align}
    0 \leq C_3 I_{\omega_0}\lr{\varphi} \leq  C_1\widetilde{I}_{\omega_0}\lr{\varphi} \leq \widetilde{I}_{\omega_0}\lr{\varphi}-\widetilde{J}_{\omega_0}\lr{\varphi}\leq C_2 \widetilde{I}_{\omega_0}\lr{\varphi} \leq C_4 I_{\omega_0}\lr{\varphi}.
  \end{align}
  Assume $\omega_\phi$ be another K\"{a}hler form in $[\omega_0]$, we have
  \begin{align}
  \label{eqn:5}
    \left|I_{\omega_\phi}\lr{\varphi-\phi}-I_{\omega_0}\lr{\varphi}\right|\leq\lr{n+1} \operatorname{OSC}\lr{\phi},
  \end{align}
  for all $\varphi\in\mathscr{H}_{X}\lr{M,\omega_0}.$
\end{mt}

We define twisted Mabuchi $K$--energy $\widetilde{\mu}_{\omega_0,\eta}$ and Ding--functional $\widetilde{F}_{\omega_0,\eta}$ on the function space $\mathscr{H}_X\lr{M,\omega_0}$ as follow:
\begin{equation*}
\widetilde{\mu}_{\omega_0,\eta}\lr{\varphi} = \frac{\im n}{2\pi V}\int_0^1\int_M e^{\theta_X\lr{\omega_{\varphi_t}}} \partial\lr{h_{\omega_{\varphi_t}} - \theta_X\lr{\omega_{\varphi_t}}} \wedge \overline{\partial}\dot{\varphi}_t \omega_{\varphi_t}^{n-1}\wedge\dd t,
\end{equation*}
and
\begin{equation*}
  \widetilde{F}_{\omega_0,\eta}\lr{\varphi} = \widetilde{J}_{\omega_0}\lr{\varphi}- \frac{1}{V}\int_M\varphi e^{\theta_X}\omega_0^n- \frac{1}{\beta}\log\lr{\frac{1}{V}\int_M e^{h_{\omega_0}-\beta\varphi}\omega_0^n}.
\end{equation*}



For convenience, we define the functional $$\hat{F}_{\omega_0}=\widetilde{J}_{\omega_0}\lr{\varphi}- \frac{1}{V}\int_M\varphi e^{\theta_X}\omega_0^n.$$

\medskip

\begin{mt}
  The functional $\widetilde{F}_{\omega_0,\eta}$,$\widetilde{\mu}_{\omega_0,\eta}$ and $\hat{F}_{\omega_0}$ are well--defined, i.e. independent of the choice of the path $\{\varphi_t\}$. Furthermore, all of them satisfy the cocycle property.
\end{mt}

\medskip

\begin{yl}
\label{lemma:2}
  For any $\varphi\in \mathscr{H}_X\lr{M,\omega_0}$,
  \begin{eqn}
  \begin{split}
  \widetilde{\mu}_{\omega_0,\eta}\lr{\varphi}=&\beta\widetilde{F}_{\omega_0,\eta}\lr{\varphi}+\frac{1}{V}\int_M \lr{h_{\omega_0}-\theta_X}e^{\theta_X}\omega_0^n\\&-\frac{1}{V}\int_M \lr{h_{\omega_\varphi}
  -\theta_X-X\lr{\varphi}}e^{\theta_X+X\lr{\varphi}}\omega_\varphi^n.
  \end{split}
  \end{eqn}
  Furthermore, we have
  \begin{eqn}
  \label{eqn:3}
  \widetilde{\mu}_{\omega_0,\eta}\lr{\varphi}\geq\beta\widetilde{F}_{\omega_0,\eta} \lr{\varphi}+\frac{1}{V}\int_M \lr{h_{\omega_0}-\theta_X}e^{\theta_X}\omega_0^n.
  \end{eqn}
\end{yl}

\medskip

\begin{yl}
\label{lemma:f}
  Assume that $\eta_1-\eta_2=\im\partial\overline{\partial}f$, for some smooth function $f$, then
  \begin{equation*}
    \left|\widetilde{\mu}_{\omega_0,\eta_1}\lr{\varphi}-\widetilde{\mu}_{\omega_0,\eta_2} \lr{\varphi}\right| \leq\operatorname{OSC}\lr{f}.
  \end{equation*}
\end{yl}

\section{Existence and uniqueness result for the twisted K\"{a}hler--Ricci soliton}
\label{section:3}
As in K\"{a}hler--Einstein and K\"{a}hler--Ricci soliton cases, finding twisted K\"{a}hler--Ricci soliton can be reduced  to solving the complex Monge--Amp\`{e}re equation \eqref{eqn:MAKRS}. To solve \eqref{eqn:MAKRS}, we use the continuity method. We consider a family of complex Monge--Amp\`{e}re equations,
\begin{equation}
\label{eqn:MAKRSt}
  \frac{\lr{\omega_0 + \im \partial \overline{\partial} \varphi}^n}{\omega_0^n} = e^{h_{\omega_0}-\beta t\varphi-\theta_X-X\lr{\varphi}}
\end{equation}
i.e.
\begin{equation}
\label{eqn:KRSt}
  Ric{\omega_\varphi}=\beta t\omega_\varphi+\lr{\beta-\beta t}\omega_0 +\eta+ L_X{\omega_\varphi}
\end{equation}
and set
\begin{equation*}
  S=\{t\in[0,1]|\text{ \eqref{eqn:MAKRSt} is solvable for }t\}.
\end{equation*}
For our convenience, we assume $\eta=(1-\beta)(\omega_0+\im \partial\overline{\partial} f_{\eta})$, where $f_{\eta}$ is a smooth function such that $\mathop{\sup}\limits_M f_\eta=0$. By \cite{zhu2000kahler} and the property of $\eta$, we know that \eqref{eqn:MAKRSt} is solvable for $t=0$, thus $S$ is not empty. If we can prove that $S$ is both open and close, then we must have $S=[0,1]$ and hence the complex Monge--Amp\`{e}re equation \eqref{eqn:MAKRS} is solvable. In the proof of the openness and closeness of $S$, we need the assumption that $\eta$ is semipositive. The key point is that the semipositivity of $\eta$ will lead to a lower bound of the Ricci curvature by a positive constant. Then we can apply the implicit function theorem to prove the openness and obtain a lower bound of the Green's function for the weighted Laplace, which is crucial to get $C^0$ estimate. We will follow methods of \cite{cao2005kahler}, \cite{tian1997kahler} and \cite{zhanggeneralized1} to obtain the openness and closeness. First, we present the following proposition for further discussion.
\begin{mt}
  Let $0<\tau\leq 1$, and suppose that \eqref{eqn:MAKRSt} is solvable at $t=\tau$. We have the following,
  \begin{enumerate}[(1).]
    \item If $\ 0<\tau<1$, there exists some $\varepsilon>0$ such that \eqref{eqn:MAKRSt} can be solvable uniquely for $t\in\lr{\tau-\varepsilon, \tau +\varepsilon} \cap \lr{0 , 1}.$
    \item $S$ is also open near $t=0$, i.e. $\exists$ a small positive number $\varepsilon$ such that there is a smooth family of solutions of \eqref{eqn:MAKRSt} for $t\in \lr{0,\varepsilon}$.
    \item If $\eta$ is strictly positive at a point, $S$ is open near $1$, i.e. \eqref{eqn:MAKRSt} is solvable for $t\in(1-\varepsilon,1]$ for some small positive $\varepsilon$.
  \end{enumerate}
\end{mt}
\begin{proof}
  For $2\leq \gamma\in \mathbb{Z}^+$ and $0<\alpha<1$, we define
  \begin{equation*}
    \mathscr{H}_X^{\gamma,\alpha}\lr{\omega_0}=\{\phi \in C^{\gamma,\alpha}(M)|\ \omega_0 + \im \partial \overline{\partial} \phi>0,\text{ and }\operatorname{Im}X\lr{\phi} =0\},
  \end{equation*}
  and
  \begin{equation*}
    \mathscr{W}_X^{\gamma,\alpha}=\{\phi \in C^{\gamma,\alpha}(M)|\ \operatorname{Im}X\lr{\phi} =0\}.
  \end{equation*}
  It is easy to see that the tangent space of $\mathscr{H}_X^{\gamma,\alpha}(\omega_0)$ is $\mathscr{W}_X^{\gamma,\alpha}$.  Consider the operator $\Psi:\mathscr{H}_X^{\gamma,\alpha}\lr{\omega_0}\times [0,1]\ra C^{\gamma-2,\alpha}\lr{M}$ defined by
  \begin{equation*}
    \Psi\lr{\varphi,t}:=\log\frac{\lr{\omega_0+\im\partial \overline{\partial}\varphi}^n}{\omega_0^n}-h_{\omega_0} + t \beta \varphi + X \lr{\varphi}.
  \end{equation*}
  The linearized operator of $\Psi$ at $\lr{t,\varphi}$ is given by
  \begin{equation*}
    L_{t,\varphi}\lr{\phi}=\laplace_{\omega_\varphi}\phi+t\beta\phi+X\lr{\phi},
  \end{equation*}
  for $\phi\in \mathscr{W}_X^{\gamma,\alpha}$. Now we prove that $L_{t,\varphi}$ is invertible. Assume $\lambda_{1,t}$ is the first eigenvalue of $L_{t,\varphi}$, and $\phi$ is an eigenfunction of $L_{t,\varphi}$ with respect to $\lambda_{1,t}$, i.e. $L_{t,\varphi}\lr{\phi}=-\lambda_{1,t}\phi$. Applying the Bochner formula and equation \eqref{eqn:MAKRSt}, we get that,
  \begin{equation}
  \label{eqn:eigenfunction}
  \begin{split}
  &\relphantom{=}\lambda_{1,t}\int_M\left|\grad_{\omega_\varphi}\phi\right|_{\omega_{\varphi}}^2 e^{\theta_X + X\lr{\varphi}} \omega_\varphi^n \\
  & =-\int_M\langle \grad_{\omega_\varphi} \lr{\laplace_{\omega_\varphi}\phi+t\beta \phi + X \lr{\phi}},  \grad_{\omega_\varphi}\phi\rangle_{\omega_\varphi}e^{\theta_X+X\lr{\varphi}} \omega_\varphi^n \\
  &=\frac{1}{2}\int_{M}\lr{\beta \lr{1-t}\omega_0 +\eta }\lr{\grad_{\omega_\varphi}\phi, J\lr{\grad_{\omega_\varphi}\phi}}e^{\theta_X+X\lr{\varphi}} \omega_\varphi^n\\
  &\repl{=}+\int_M\left| \grad^{(1,0)}_{\omega_\varphi} \grad^{(1,0)}_{\omega_\varphi} \phi \right|_{\omega_\varphi}^2 e^{\theta_X+X\lr{\varphi}} \omega_\varphi^n
  \end{split}
  \end{equation}
  In the case $0\leq\tau<1$, we have
  \begin{equation*}
  \begin{split}
  &\relphantom{=}\lambda_{1,\tau}\int_M\left|\grad_{\omega_\varphi}\phi \right|_{\omega_{\varphi}}^2 e^{\theta_X + X\lr{\varphi}} \omega_\varphi^n\\
  &\geq\frac{\lr{1-\tau}\beta}{2}\int_M\omega_0\lr{\grad_{\omega_\varphi}\phi, J\lr{\grad_{\omega_\varphi}\phi}} e^{\theta_X + X\lr{\varphi}} \omega_\varphi^n\\
  &>0,
  \end{split}
  \end{equation*}
  which implies that $\lambda_{1,\tau}>0$, i.e. $L_{\tau,\varphi}$ is invertible. And consequently, the first and second statements of the proposition hold.

  When $\tau=1$, if $\lambda_{1,\tau}=0$, then
  \begin{equation*}
    \frac{1}{2}\int_{M}\eta\lr{\grad_{\omega_\varphi}\phi, J\lr{\grad_{\omega_\varphi}\phi}}e^{\theta_X+X\lr{\varphi}} \omega_\varphi^n+\int_M\left| \grad^{(1,0)}_{\omega_\varphi} \grad^{(1,0)}_{\omega_\varphi} \phi \right|_{\omega_\varphi} ^2 e^{\theta_X+X\lr{\varphi}} \omega_\varphi^n=0.
  \end{equation*}
  Since $\eta$ is smooth and strictly positive at some point $p$, we get that $\grad^{(1,0)}_{\omega_\varphi}\phi=0$ in a neighborhood of $p$, which implies that $\grad^{(1,0)}_{\omega_\varphi}\phi=0.$ So $L_{1,\varphi}$ is invertible. And the last statement holds.
\end{proof}

\medskip

Next, we prove the closeness of $S$. Let $\{\varphi_t \}$ be a smooth family of solution of \eqref{eqn:MAKRSt} for $t\in(0,1]$. Along the path $\{\varphi_t\}$, we have
\begin{align}
\label{eqn:h-btvar=V}
  \int_M e^{ h_{\omega_0} - \beta t \varphi_t}\omega_0^n = \int_M e^{ \theta_X + X \lr{ \varphi_t}}\omega_{ \varphi_t}^n = V.
\end{align}
Differentiating \eqref{eqn:h-btvar=V} with respect to $t$, we have
\begin{equation}
\label{eqn:6}
  \int_M \varphi_t e^{ h_{\omega_0} - \beta t \varphi_t}\omega_0^n = - \int_M t \dot{\varphi}_t e^{ h_{\omega_0} - \beta t \varphi_t}\omega_0^n.
\end{equation}

\medskip

\begin{mt}
\label{prop:2.2}
Let $\{\varphi_t\}$ be a smooth family of solution of \eqref{eqn:MAKRSt} for $t\in(0,1]$. Then
\begin{equation*}
  \hat{ F}_{ \omega_0} \lr{\varphi_t} = - \frac{1} {t}\int_0^t \lr{ \widetilde{I}_{\omega_0} \lr{ \varphi_s}- \widetilde{J}_{\omega_0} \lr{ \varphi_s}}\dd t.
\end{equation*}
\end{mt}

\medskip

Furthermore, we have:

\medskip

\begin{yl}
\label{lemma:1}
  Let $\{\varphi_t\}$ be a smooth family of solution of \eqref{eqn:MAKRSt} for $t\in(0,1]$. We have
  \begin{equation*}
    \frac{\dd}{\dd t}\lr{\widetilde{I}_{\omega_0} \lr{ \varphi_t}- \widetilde{J}_{\omega_0} \lr{ \varphi_t}}\geq 0,
  \end{equation*}
  i.e. $\widetilde{I}_{\omega_0} \lr{ \varphi_t}- \widetilde{J}_{\omega_0} \lr{ \varphi_t}$ is nondecreasing with respect to $t$ along $\{\varphi_t\}$.
\end{yl}

\medskip

Before proving the closeness of $S$, we recall two useful estimates which have been applied in \cite{cao2005kahler} and \cite{zhu2000kahler} .

\medskip

\begin{yl}[\cite{cao2005kahler}]
\label{lemma:6}
Let $\varphi\in \mathscr{H}_X\lr{M,\omega_0}$. Suppose that
\begin{equation*}
  Ric\lr{\omega_\varphi}-L_X\omega_{\varphi}\geq\lambda\omega_\varphi,
\end{equation*}
and
\begin{equation*}
\laplace_{\omega_\varphi}\theta_X\lr{\omega_\varphi}\leq k,
\end{equation*}
for some positive number $\lambda$ and $k$. Then, there are uniform constants $C_1$, $C_2$ depending only on $\lambda$ and $k$ such that the Green function $G$ with respect to the operator $\laplace_{\omega_{\varphi}}+X$ is bounded from below by $C_1$ and the following estimate for $\mathop{\sup}_M\lr{-\varphi}$ holds,
\begin{equation*}
  \sup_M\lr{-\varphi}\leq\frac{1}{V}\int_M \lr{ -\varphi} e^{\theta_X\lr{\omega_\varphi}}\omega_\varphi^n+C_2.
\end{equation*}
\end{yl}

\medskip

\begin{yl}[\cite{zhu2000kahler}]
\label{lemma:3}
  For any $\varphi\in\mathscr{H}_X\lr{M,\omega_0}$, there exists a uniform constant $C$ independent of $\varphi$, such that $\left|X\lr{\varphi}\right|\leq C.$
\end{yl}

Next, we consider the closeness of $S$.

\medskip

\begin{mt}
\label{mt:2.6}
  Let $\varphi=\varphi_t\lr{t\geq t_0 >0}$ be a solution of \eqref{eqn:MAKRSt} at $t$. Suppose that $\widetilde{\mu}_{\omega_0,\eta}$(or $\widetilde{F}_{\omega_0,\eta}$) is  proper, then, the $C^0$--norm of $\varphi$ is bounded depending only on $X$, $t_0$, the properness and the geometry of $\lr{M,\omega_0}$.
\end{mt}

\medskip

\begin{proof}
  Following from Lemma \ref{lemma:2}, we just prove the case when $\widetilde{\mu}_{\omega_0,\eta}$ is proper. To do this we first give an estimate of the oscillation of $\varphi$.

  Let $\theta_X'=\theta_X\lr{\omega_\varphi}=\theta_X+X\lr{\varphi}$. By the equation
  \begin{equation*}
    L_X Ric\lr{\omega}=-\im\partial\overline{\partial}\laplace_\omega\theta_X\lr{\omega},
  \end{equation*}
  and the maximal principle, we have
  \begin{equation*}
    \laplace_{\omega_\varphi}\theta_X'=-\beta\theta_X'-(1-\beta)\theta_X -(1-\beta)X(f_\eta )-X\lr{h_{\omega_{\varphi_t}}}+c_t
  \end{equation*}
  for some constant $c_t$. By using the maximal principle to \eqref{eqn:KRSt}, we have
  \begin{equation*}
    h_{\omega_\varphi}=\theta_X'-\lr{1-t}\beta\varphi+c_t',
  \end{equation*}
  where $c_t'$ is some constant, which implies that
  \begin{equation}
  \label{eqn:9}
    \laplace_{\omega_\varphi}\theta_X'=-\beta\theta_X' - (1-\beta)\theta_X -(1-\beta) X(f_\eta)-X\lr{\theta_X'}+\lr{1-t}\beta X\lr{\varphi}+c_t.
  \end{equation}
  Applying the maximal principle and \eqref{eqn:9}, we get that
  \begin{equation}
  \label{eqn:10}
    c_t\leq \beta||\theta_X'||_{C^0} +(1-\beta)||\theta_X||_{C^0} +(1-\beta)||X(f_\eta)||_{C^0} +\beta\lr{1-t}||X\lr{\varphi}||_{C^0}.
  \end{equation}
  So we have
  \begin{equation*}
  \begin{split}
    \laplace_{\omega_\varphi}\theta_X'&\leq -||X||^2_{\omega_\varphi} + (2+\beta(1-t))||\theta_X||_{C^0} +(2\beta+2\beta(1-t)) ||X(\varphi)||_{C^0}\\
    &\repl{\leq}+2(1-\beta)||X(f_\eta)||_{C^0}\\
    &\leq k
  \end{split}
  \end{equation*}
  for some uniform constant $k$ by \eqref{eqn:9}, \eqref{eqn:10} and Lemma \ref{lemma:1}.

  Furthermore, we have that
  \begin{equation*}
    Ric\lr{\omega_\varphi}-L_X\omega_\varphi=\beta t\omega_\varphi+\beta\lr{1-t}\omega_0+\eta\geq \beta t_0\omega_\varphi.
  \end{equation*}

  By Lemma \ref{lemma:6} with $\lambda=\beta t_0$, we get that
  \begin{equation}
  \label{eqn:11}
    \sup_M\lr{-\varphi}\leq\frac{1}{V}\int_M \lr{ -\varphi} e^{\theta_X'}\omega_\varphi^n+C_1',
  \end{equation}
  for some uniform constant $C_1'$ depending only on $X$, $t_0$.

  On the other hand, by using the Green formula for the Laplace with respect to $\widetilde{\omega}$ satisfying that
  \begin{equation*}
  \widetilde{\omega}^n=e^{\theta_X}\omega_0^n,
  \end{equation*}
  we get that
  \begin{equation}
  \label{eqn:12}
    \sup_M\varphi\leq \frac{1}{V}\int_M\varphi e^{\theta_X}\omega_0^n+C_2'
  \end{equation}
  for some uniform constant $C_2'$(cf. Lemma 5.3 in \cite{tian2000nonlinear}). Hence, combining \eqref{eqn:11} and \eqref{eqn:12}, we have
  \begin{equation*}
    \operatorname{OSC}_M\varphi\leq \frac{1}{V}\int_M\varphi\lr{e^{\theta_X}\omega_0^n-e^{\theta_X'}\omega_\varphi^n }+C_3',
  \end{equation*}
  for some uniform constant $C_3'$.

  Next we show that $$\widetilde{I}_{\omega_0}\lr{\varphi}=\frac{1}{V}\int_M\varphi\lr{e^{\theta_X}\omega_0^n- e^{\theta_X'}\omega_\varphi^n}$$ is uniformly bounded from above under the condition that $\widetilde{\mu}_{\omega_0,\eta}$ is proper.

  By the equation \eqref{eqn:MAKRSt} and the definition of $h_\omega$, we get that
  \begin{equation}
  \label{eqn:13}
  \begin{split}
    h_{\omega_\varphi}&=h_{\omega_0}-\log\frac{\omega_\varphi^n}{\omega_0^n}-\beta\varphi- \log\lr{\frac{1}{V} \int_M e^{h_{\omega_0}-\beta\varphi}\omega_0^n}\\
    &=h_{\omega_0}-\lr{h_{\omega_0}-\beta t\varphi -\theta_X'}-\beta\varphi- \log\lr{\frac{1}{V} \int_M e^{h_{\omega_0}-\beta\varphi}\omega_0^n}\\
    &=\theta_X'-\beta\lr{1-t}\varphi- \log\lr{\frac{1}{V} \int_M e^{h_{\omega_0}-\beta\varphi}\omega_0^n}.
  \end{split}
  \end{equation}
  Proposition \ref{prop:1}, Proposition \ref{prop:2.2} and equation \eqref{eqn:13} imply that
  \begin{align*}
    \widetilde{\mu}_{\omega_0,\eta}\lr{\varphi}&=\beta \widetilde{F}_{\omega_0,\eta}\lr{ \varphi} + \frac{1}{V}\int_M \lr{h_{\omega_0}-\theta_X}e^{\theta_X}\omega_0^n -\frac{1}{V}\int_M \lr{ h_{ \omega_{ \varphi}} - \theta_X '}e^{\theta_X'}\omega_\varphi^n\\
    &=\beta \hat{F}_{\omega_0}\lr{\varphi} +\frac{\beta \lr{1-t}}{V}\int_M \varphi e^{\theta_X'}\omega_\varphi^n + \frac{1}{V} \int_M \lr{h_{\omega_0}-\theta_X }e^{\theta_X}\omega_0^n\\
    &=\beta\lr{1-t}\lr{\widetilde{J}_{\omega_0}\lr{\varphi}-\widetilde{I}_{\omega_0}\lr{\varphi}} +\beta t \hat{F}_{\omega_0}\lr{\varphi} + \frac{1}{V} \int_M \lr{h_{\omega_0}-\theta_X }e^{\theta_X}\omega_0^n\\
    &\leq \frac{1}{V} \int_M \lr{h_{\omega_0}-\theta_X }e^{\theta_X}\omega_0^n\\
    &\leq C_4',
  \end{align*}
  $\widetilde{I}_{\omega_0}\lr{\varphi}$ is uniform bounded from above, since $\widetilde{\mu}_{\omega_0,\eta}$ is proper, and Proposition \ref{prop:1}. So we get the $C^0-$estimate for the twisted K\"{a}hler--Ricci soliton.
\end{proof}

\medskip

By Yau's estimate \cite{yau1978ricci} or Zhu's estimate \cite{zhu2000kahler} for complex Monge--Amp\`{e}re equations, the $C^0$--estimate implies the $C^{2,\alpha}$--estimate and the elliptic Schauder estimates give the higher order estimates. We get the closeness of $S$, and the existence of twisted K\"{a}hler--Ricci soliton with respect to $\eta$.

\begin{zj}
  As a consequence of the argument above and the argument of Bando-Mabuchi, we can get easily that if the modified K-energy $\tilde{\mu}_{\omega,X}$ in \cite{cao2005kahler} is bounded from below, then $R(X)=1$.
\end{zj}

Following the idea of \cite{berndtsson2013brunn}, we consider the uniqueness of the twisted K\"ahler-Ricci soliton with respect to a smooth semipositive $(1,1)$--form $\eta$ which is strictly positive at one point.

\begin{thm}
\label{Theorem:unique twisted}
  If $\eta$ is a smooth semipositive $(1,1)$--form cohomology to $(1-\beta)\omega_0$ and $\eta$ is strictly positive at one point, the solution to \eqref{eqn:MAKRS} is unique.
\end{thm}
\begin{proof}
  We assume that $\omega_1=\omega_0+\im \partial\overline{\partial}\varphi_1$ and $\omega_2=\omega_0+ \im \partial\overline{\partial}\varphi_2$ are two different solutions to \eqref{eqn:MAKRS}. It is easy to see that $\varphi_1$ and $\varphi_2$ are stationary points of $\widetilde{F}_{\omega_0,\eta}$. According to page 32 of \cite{berndtsson2013brunn}, we assume that $\{\phi_t|0\leq t\leq 1\}$ is a curve of $C^1$--geodesics connecting $\varphi_1$ and $\varphi_2$, and $\operatorname{Im}X(\phi_t)=0$.

  Direct computation shows that $\hat{F}_{\omega_0}(\phi_t)$ is linear with respect to $t$. Theorem 1.1 of \cite{berndtsson2013brunn} shows that $$-\frac{1}{\beta}\log(\frac{1}{V}\int_M e^{h_{\omega_0}-\beta\phi_t}\omega_0^n)$$ is convex with respect to $t$. So is $\widetilde{F}_{\omega_0,\eta}(\phi_t)$. Since $\varphi_1$ and $\varphi_1$ are stationary points of $\widetilde{F}_{\omega_0,\eta}$, we get that $\widetilde{F}_{\omega_0,\eta}(\phi_t)$ is actually linear with respect to $t$. So $-\frac{1}{\beta}\log(\frac{1}{V}\int_M e^{h_{\omega_0}-\beta\phi_t}\omega_0^n)$ is linear in $t$. Theorem 6.2 of \cite{berndtsson2013brunn} implies that there exists a holomorphic vector field $V_t$ on $M$ with flow $F_t$ such that
  \begin{align}
    F_t^*(\omega_0+\im \partial\overline{\partial}\phi_t)&=\omega_0+ \im \partial\overline{\partial} \phi_0 =\omega_1\label{eqn:F_t}\\
    i_{V_t}\eta&=0.
  \end{align}
  However, the semi-positiveness and strict positiveness of $\eta$ at some point implies that $V_t=0$. So $F_t$ is an identity. The equation \eqref{eqn:F_t} shows that $\omega_1=\omega_2$, which is a contradiction.
\end{proof}

\section{A Moser--Trudinger type inequality}
\label{section:4}
Let $\omega_0\in 2\pi c_1(M)$ be a K\"{a}hler form on $M$, satisfying that
\begin{equation*}
  \begin{cases}
    Ric\lr{\omega_0}-L_X\omega_0\geq \lr{\beta-\varepsilon}\omega_0+\eta\\
    \left|X\lr{h_{\omega_0}-\theta_X}\right|\leq \varepsilon c_1
  \end{cases}
\end{equation*}
for some positive number $\varepsilon$ and $c_1$, where $\theta_X=\theta_X\lr{\omega_0}$. Similar to \cite{cao2005kahler}, we define $a\lr{\omega}$ to be the maximal constant such that,
\begin{equation*}
  \forall 0<r<1,\text{ and } x\in M, \int_{B_{r}\lr{x}}e^{\theta_X \lr{ \omega}}\omega^n\geq a\lr{\omega}r^{2n},
\end{equation*}
where $B_{r}\lr{x}$ is the geodesic ball in $M$ of radio $r$ centering $x$ with respect to $\omega$. And
\begin{equation*}
\lambda_1\lr{\omega}=\inf_{\substack{v\in C^\infty\lr{M}\\ \grad_{\omega} v \neq 0}} \frac{\int_M \left| \grad_{\omega} v\right|_{\omega}^2 e^{\theta_X\lr{\omega}}\omega^n}{\int_M v^2 e^{\theta_X\lr{\omega}}\omega^n - \lr{\int_M v e^{\theta_X\lr{\omega}}\omega^n}^2},
\end{equation*}
is the first eigenvalue of the operator $\laplace_\omega+X$. For two K\"{a}hler forms $\omega_1$, $\omega_2\in \mathscr{K}_X\lr{\omega_0}$ satisfying
\begin{equation*}
\frac{1}{2}\omega_2\leq \omega_1\leq 2\omega_2,
\end{equation*}
we have the following estimates
\begin{equation}
\label{eqn:27}
\lambda_1\lr{\omega_1}\geq 2^{-2n-1}e^{-2C\lr{\omega_0,X}}\lambda_1\lr{\omega_2},
\end{equation}
\begin{equation}
\label{eqn:28}
a\lr{\omega_1}\geq 2^{-2n}e^{-C\lr{\omega_0,X}}a\lr{\omega_2},
\end{equation}
where $C\lr{\omega_0,X}$ is the constant appeared in Lemma \ref{lemma:3}. Similar to \cite{cao2005kahler}, \cite{tian1997kahler} and \cite{zhanggeneralized1}, we have the smoothing lemma as follow:

\begin{yl}[Smoothing lemma]
$\omega_1$ is the solution of $$\frac{\partial \omega_t}{\partial t}=-Ric\lr{ \omega_t} +\eta +\beta \omega_t +L_X\omega_t$$ at $t=1$ with initial data $\omega_0$. And $\omega_t = \omega_0 +\im \partial \overline{\partial}u_t$, $u_0=0$. The following inequalities hold:
\begin{enumerate}[1.]
  \item $\left|\left|u_1\right|\right|_{C^0\lr{M}}\leq \frac{e^\beta}{\beta}\left|\left| h_{\omega_0}-\theta_X \right|\right|_{C^0\lr{M}},$
  \item $\left|\left| h_{\omega_1}-\theta_X\lr{\omega_1} \right|\right|_{C^{0,\frac{1}{2}} \lr{\omega_1}} \leq 4 C\lr{n,c_1,a\lr{\omega_1},\lambda_1\lr{\omega_1}}\lr{1+\left|\left| h_{\omega_0}-\theta_X \right|\right|_{C^0\lr{M}}}\varepsilon^{\frac{1}{4\lr{n+1}}},$
\end{enumerate}
where $C\lr{n,c_1,a\lr{\omega_1},\lambda_1\lr{\omega_1}}=e^\beta \lr{1+ \sqrt{\frac{2V\lr{c_1+n}}{a\lr{\omega_1}\lambda_1\lr{\omega_1}}}}$.
\end{yl}

Using the smoothing lemma and discussion similar to \cite{cao2005kahler}, \cite{phong2008moser} and \cite{zhanggeneralized1}, we can establish a Moser--Trudinger type inequality for the twisted Ding--functional $\widetilde{F}_{\omega_{TKS},\eta}$ or $K$--energy $\widetilde{\mu}_{\omega_{TKS},\eta}$.
Specially, the Moser-Trudinger inequality here is linear form. And for the readers' convenience, we give the details of the proof:
\begin{thm}
\label{theorem:3}
  Let $\lr{M,\omega_0}$ be a K\"{a}hler manifold and $L_{\lr{\IM X}}\omega_0=0$ where $X$ is a holomorphic vector field on $M$. Assuming  that $\eta\in \lr{1-\beta}\mathscr{K}^0_X\lr{\omega_0}$ is strictly positive at one point. If there exists a twisted K\"{a}hler--Ricci soliton metric $\omega_{TKS}\in \mathscr{K}_X\lr{\omega_0}$, then there exist uniform positive constants $C_1$, $C_2$ depending only on $\eta$, $\beta$ and the geometry of $\lr{M,\omega_{TKS}}$ such that
  \begin{equation*}
    \widetilde{F}_{\omega_{TKS},\eta}\lr{\phi}\geq C_1 J_{\omega_{TKS}}\lr{\phi}-C_2
  \end{equation*}
  for all $\phi\in \mathscr{H}_{X}\lr{M,\omega_{TKS}}$.
\end{thm}

\begin{proof}
  For any $\phi\in\mathscr{H}_X\lr{M,\omega_{TKS}}$, we denote $\omega_g=\omega_{TKS} + \im \partial \overline{\partial}\phi$. We consider the following Monge--Amp\`{e}re equations with parameter $t\in [0,1]$ with the notions $\theta_X=\theta_X\lr{\omega_g}$:
  \begin{equation}
  \label{eqn:TKRSt}
    \begin{cases}
      \frac{\lr{\omega_g+ \im \partial\overline{\partial} \varphi}^n}{\omega_g^n} = h_{\omega_g} - \theta_X - X\lr{\varphi} -\beta t \varphi\\
      \omega_g + \im \partial\overline{\partial} \varphi>0.
    \end{cases}
  \end{equation}
  Clearly, $-\phi$ is a solution of \eqref{eqn:TKRSt} at $t=1$ modulo a constant. Let
  \begin{equation*}
    T=\{t\in[0,1]|\text{ \eqref{eqn:TKRSt} is solvable at }t\}.
  \end{equation*}
  By the assumption that $\eta$ is strictly positive at one point and the fact that $\widetilde{I}_{\omega_g} \lr{\varphi_t} - \widetilde{J}_{\omega_g} \lr{\varphi_t}$ is nondecreasing in $t$, we get that $(0,1]\subset T$, i.e. for any $t\in (0,1]$, there is a solution $\varphi_t$ for \eqref{eqn:TKRSt} at $t$. We assume that $\varphi_t$ is the solution of \eqref{eqn:KRSt} at $t$, with $\varphi_1=-\phi+C$ and denote that $\omega_t=\omega_g+\im \partial \overline{\partial} \varphi_t$. Consequently, the $C^3$--norm of $\varphi_t$ for $t\geq \frac{1}{2}$ is uniformly bounded by $\phi$, $\omega_{TKS}$, $n$, $\eta$, and $X$.

  Since $\omega_t$ is a solution to $$Ric\lr{\omega_t}=\beta t\omega_t +\lr{\beta-\beta t}\omega_g +\eta + L_X\omega_t,$$ the maximal principle implies that
  \begin{equation*}
    h_{\omega_t}-\theta_X\lr{\omega_t}=-\beta\lr{1-t} \varphi_t+c_t',
  \end{equation*}
  where $c_t'$ is determined by
  \begin{equation*}
    \int_M e^{c_t'- \beta\lr{1-t}\varphi_t+ \theta_X\lr{\omega_t}}\omega_t^n=\int_M e^{h_{\omega_t}} \omega_t^n =\int_M \omega_t^n =\int_M e^{\theta_X\lr{\omega_t}}\omega_t^n.
  \end{equation*}
  In particular, we get that $c_t'-\beta\lr{1-t}\varphi_t$ changes sign, which implies that $\left|c_t'\right|\leq \beta\lr{1-t}\left|\left|\varphi_t\right|\right|_{C^{0}\lr{M}}$, and consequently,
  \begin{equation*}
    \left|\left|h_{\omega_t}-\theta_X\lr{\omega_t}\right|\right|_{C^{0}\lr{M}}\leq 2\beta\lr{1-t}\left|\left|\varphi_t\right|\right|_{C^{0}\lr{M}}.
  \end{equation*}
  On the other hand, the K\"{a}hler metric $\omega_t$ satisfies that
  \begin{equation*}
    \begin{cases}
      Ric{\omega_t}-L_X\omega_t=\beta t\omega_t+\eta+\beta\lr{1-t}\omega_g\geq \beta t\omega_t+\eta\\
      \left|X\lr{h_{\omega_t}-\theta_X\lr{\omega_t}}\right| =\beta\lr{1-t}\left|X\lr{\varphi_t}\right|\leq 2\beta\lr{1-t} C\lr{\omega_{TKS},X}.
    \end{cases}
  \end{equation*}
  With $\varepsilon=\beta-\beta t$, $c_1=2C\lr{\omega_{TKS},X}$, we know that $\omega_t$ satisfies that
  \begin{equation}
  \label{eqn:23}
  \begin{cases}
    Ric\lr{\omega_t}-L_X\omega_t\geq \lr{\beta-\varepsilon}\omega_t+\eta\\
    \left|X\lr{h_{\omega_t}-\theta_X\lr{\omega_t}}\right|\leq \varepsilon c_1.
  \end{cases}
\end{equation}
Let $\omega_t$ be the initial metric in the flow we have considered in last section. Applying the smoothing lemma, we obtain a K\"{a}hler form $\omega_t'=\omega_t+\im \partial\overline{\partial}u_t$ satisfying that
\begin{equation*}
  \left|\left|u_t\right|\right|_{C^0\lr{M}}\leq \frac{e^\beta}{\beta}\left| \left| h_{\omega_t} - \theta_X \lr{\omega_t} \right| \right|_{C^{0}\lr{M}} \leq 2e^\beta\lr{1-t}\left|\left|\varphi_t\right|\right|_{C^0\lr{M}}
\end{equation*}
\begin{equation*}
  \begin{split}
    &\repl{\leq}\left|\left|h_{\omega_t'}-\theta_X\lr{\omega_t'}\right|\right|_{C^{0,\frac{1}{2}}\lr{\omega_t'}}\\
    &\leq 4 C\lr{n, c_1, a\lr{\omega_t'}, \lambda_1\lr{\omega_t'}}(1+ \left| \left| h_{\omega_t}-\theta_X\lr{\omega_t} \right| \right|_{C^{0}\lr{M}})\lr{\beta -\beta t}^{\frac{1}{4n+4}}\\
    &\leq 8C\lr{n, c_1, a\lr{\omega_t'}, \lambda_1\lr{\omega_t'}}(1+ \lr{1-t}\left| \left|\varphi_t \right| \right|_{C^{0}\lr{M}})\lr{\beta -\beta t}^{\frac{1}{4n+4}}.
  \end{split}
\end{equation*}

As before, there exists $\varphi_t'$ such that $\omega_{TKS}=\omega_t'+\im \partial \overline{\partial} \varphi_t$ and
\begin{equation*}
  -\log\frac{\omega_{TKS}^n}{\omega_{t}'^n}=-h_{\omega_t'}+\beta\varphi_t'+\theta_X\lr{\omega_t'} +X\lr{\varphi_t'}.
\end{equation*}
It follows from the maximal principle that
\begin{equation}
\label{eqn:29}
  \varphi_t=\varphi_1-\varphi_t'+u_t+\mu_t,
\end{equation}
where $\mu_t$ is a constant. The normalization $\int_Me^{h_{\omega_t'}}\omega_t'^n=\int_M\omega_t'^n$ implies that
\begin{equation*}
  \left|\mu_t\right|\leq2\lr{1+e^\beta}\lr{1-t}\left|\left|\varphi_t\right|\right|_{C^0\lr{M}}.
\end{equation*}

Next, we give an estimate of $\left|\left|\varphi_t'\right|\right|_{C^0\lr{M}}$ while $t$ is sufficiently closed to $1$. We denote the operator $$\Xi:\mathscr{H}_X^{2,\frac{1}{2}}\lr{M,\omega_{TKS}}\times[0,1]\ra C^{0,\frac{1}{2}}\lr{M}$$ by
\begin{equation*}
  \Xi\lr{\varphi',t}=\log\frac{\lr{\omega_{TKS}-\im \partial\overline{\partial} \varphi'}^n}{\omega_{TKS}^n}+h_{\omega_t'}-\beta\varphi'- \theta_X\lr{\omega_t'} -X\lr{\varphi'}.
\end{equation*}
Obviously, $\Xi\lr{0,1}=0$, and the linearization operator of $\Xi$ at $\varphi'=0$ is $$- \laplace_{\omega_{TKS}}-X-\beta,$$ which is invertible under the assumption that $\eta$ is strictly positive at one point. So $\Xi$ is invertible for $t$ sufficiently closed to $1$, i.e. $\exists \sigma>0$ if
\begin{equation}
\label{eqn:25}
  \left|\left|h_{\omega_t'}-\theta_{X}\lr{\omega_t'}\right|\right|_{C^{0, \frac{1}{2}}\lr{\omega_{TKS}}}\leq \sigma
\end{equation}
then,
\begin{equation}
\exists! \varphi_t', \Xi{\varphi_t'}=0,\text{ and }\left|\left| \varphi_t' \right|\right|_{C^{2,\frac{1}{2}} \lr{\omega_{TKS}}}\leq C_0\sigma,
\end{equation}
where $C_0=C\lr{n, c_1, 2^{-2n}e^{-C\lr{X,\omega_{TKS}}}a\lr{\omega_{TKS}}, 2^{-2n-1}e^{-2C\lr{X,\omega_{TKS}}}\lambda_1\lr{\omega_{TKS}}}+1$ with notations that $c_1$ is the constant in \eqref{eqn:23} and $C(X,\omega_{\omega_{TKS}})$ in Lemma \ref{lemma:6}. And we choose $\sigma$ small such that $C_0\sigma<\frac{1}{4}$ and $\frac{1}{\beta}\lr{\frac{\sigma}{32C_0}}^{4n+4}\leq \frac{1}{12e^\beta}$.

For convenience, we do some computation to $\widetilde{F}_{\omega_{TKS},\eta}\lr{\phi}$:
\begin{equation*}
  \begin{split}
    \widetilde{F}_{\omega_{TKS},\eta}\lr{\phi}&= -\widetilde{F}_{\omega_g,\eta} \lr{\varphi_1} =-\hat{F}_{\omega_g}\lr{\varphi_1}\\
    &= \int_0^1\lr{\widetilde{I}_{\omega_g}\lr{\varphi_t} - \widetilde{J}_{\omega_g}\lr{\varphi_t}}\dd t\\
    &\geq c\lr{1-t}\widetilde{J}_{\omega_g}\lr{\varphi_t}.
  \end{split}
\end{equation*}
\begin{equation*}
  \begin{split}
    \widetilde{J}_{\omega_g}\lr{\varphi_t} - \widetilde{J}_{\omega_g}\lr{\varphi_1}&= \int_M \varphi_t e^{\theta_X}\omega_g^n-\int_M \varphi_1e^{\theta_X}\omega_g^n + \hat{F}_{\omega_g}\lr{\varphi_t} - \hat{F}_{\omega_g}\lr{\varphi_1}\\
    &= \int_M \lr{ \varphi_t- \varphi_1} e^{\theta_X}\omega_g^n+ \hat{F}_{\omega_{TKS}}\lr{\varphi_t -\varphi_1}\\
    &=\int_M \lr{ \varphi_t- \varphi_1} \lr{e^{\theta_X} \omega_g^n- e^{\theta_X \lr{\omega_{\varphi_1}}} \omega_{\varphi_1}^n}+ \widetilde{J}_{\omega_{TKS}} \lr{\varphi_t-\varphi_1}\\
    &\geq \int_M \lr{ \varphi_t- \varphi_1} \lr{e^{\theta_X} \omega_g^n- e^{\theta_X \lr{\omega_{\varphi_1}}} \omega_{\varphi_1}^n}\\
    &\geq -\operatorname{OSC}_M\lr{\varphi_t-\varphi_1}.
  \end{split}
\end{equation*}
Hence, $\widetilde{J}_{\omega_g}\lr{\varphi_t}\geq c I_{\omega_g}\lr{\varphi_1} - \operatorname{OSC}_M\lr{\varphi_t-\varphi_1} $. Consequently,
\begin{equation}
\label{eqn:26}
\begin{split}
\widetilde{F}_{\omega_{TKS},\eta} \lr{\phi} &\geq c\lr{1-t} I_{\omega_g}\lr{\varphi_1}- C\lr{1-t} \operatorname{OSC}_M \lr{\varphi_t -\varphi_1}\\
&=c\lr{1-t} I_{\omega_{TKS}}\lr{\phi}- C\lr{1-t}\operatorname{OSC}_M \lr{\varphi_t -\varphi_1}.
\end{split}
\end{equation}

To prove the inequality required, we will discuss case by case.

\textbf{Case1:} Existing $t_0$, such that $\lr{\lr{1-t_0}\beta}^{\frac{1}{4 \lr{n+1}}} \leq \frac{\sigma}{32C_0}$, and
\begin{equation*}
  \forall\text{ }t\in(t_0,1],\text{ }\lr{1-t}\left|\left|\varphi_t\right|\right|_{C^{0}\lr{M}}\lr{\lr{1-t}\beta}^{\frac{1}{4\lr{n+1}}} < \frac{\sigma}{32C_0},
\end{equation*}
\begin{equation*}
  \lr{1-t_0}\left|\left|\varphi_{t_0}\right|\right|_{C^{0}\lr{M}}\lr{\lr{1-t_0}\beta}^{\frac{1}{4\lr{n+1}}} =\frac{\sigma}{32C_0}.
\end{equation*}
Then we claim that $\left|\left|\varphi_t'\right|\right|_{C^{2,\frac{1}{2}}}\lr{\omega_{TKS}}< \frac{1}{4}$ for $t\in(t_0,1]$. If not, existing $t_1 \in (t_0,1]$ such that $\left|\left| \varphi_{t_1}' \right|\right|_{C^{2,\frac{1}{2}}}\lr{\omega_{TKS}}=  \frac{1}{4}$. So we have $\frac{1}{2} \omega_{TKS} \leq \omega_{t_1}' \leq 2 \omega_{TKS}.$ Associated with the inequality above, we get that:
\begin{equation*}
\left|\left|h_{\omega_{t_1}'} -\theta_{X} \lr{\omega_{t_1}'} \right|\right|_{C^{0,\frac{1}{2}} \lr{\omega_{TKS}}}< \sigma
\end{equation*}
But \eqref{eqn:25} implies that $\left|\left| \varphi_{t_1}'\right|\right|_{C^{2,\frac{1}{2}}\lr{\omega_{TKS}}}\leq C_0 \sigma <\frac{1}{4}$, which is a contradiction. The equation \eqref{eqn:29} implies that for all $t\in[t_0,1]$
\begin{equation}
\label{eq:1}
\begin{split}
\left|\left|\varphi_t-\varphi_1\right|\right|_{C^0\lr{M}} &\leq \frac{1}{4}+\left|\left|u_t\right|\right|_{C^0\lr{M}}+\left|\mu_t\right|\\
&\leq \frac{1}{4}+6e^\beta\lr{1-t}\left|\left|\varphi_t\right|\right|_{C^0\lr{M}},
\end{split}
\end{equation}
Since $1-t\leq 1-t_0\leq\frac{1}{12e^\beta}$, then we have that
\begin{equation*}
  \frac{1}{2}\left|\left|\varphi_t\right|\right|_{C^0\lr{M}} - \frac{1}{4} \leq \left|\left|\varphi_1\right|\right|_{C^0\lr{M}} \leq 2 \left|\left| \varphi_t \right|\right|_{C^0\lr{M}}+\frac{1}{4}.
\end{equation*}
Further more, associated with \eqref{eqn:26} we can get that
\begin{equation*}
\label{eqn:31}
\begin{split}
  \widetilde{F}_{\omega_{TKS},\eta}\lr{\phi}&\geq c\lr{1-t} I_{\omega_{TKS}} \lr{\phi}-Ce^\beta(24\left|\left| \varphi_1 \right|\right|_{C^0\lr{M}}- 6 )\lr{1-t}^2 - \frac{C}{2}\lr{1-t}\\
  &\geq c\lr{1-t} I_{\omega_{TKS}} \lr{\phi}-24Ce^\beta \lr{1-t}^2\operatorname{OSC}_M\phi - C\lr{1-t}.
\end{split}
\end{equation*}

In case $\operatorname{OSC}_M\phi\leq \hat{C}\lr{1+I_{\omega_{TKS}}\lr{\phi}}$, we get $\widetilde{F}_{\omega_{TKS},\eta}\lr{\phi}\geq c\lr{1-t}I_{\omega_{TKS}}\lr{\phi}-C$ for some uniform constant $c$ and $C$. Now choosing $t=t_0$, by the definition of $t_0$, we get that $$\lr{1-t_0}^{\frac{4n+5}{4n+4}}I_{\omega_{TKS}}\lr{\phi}\geq \frac{\beta^{4n+4}\sigma}{32C_0}-\lr{\frac{1}{2}+2\hat{C}}\lr{\frac{\beta^{4n+4}\sigma}{32C_0}}^{ \frac{4n+5}{4n+4}}>C\lr{\sigma}>0,$$
therefore,
\begin{equation*}
  \widetilde{F}_{\omega_{TKS},\eta}\lr{\phi}\geq CI_{\omega_{TKS}}^{\frac{1}{4n+5}}\lr{\phi}-C'.
\end{equation*}

By the $C^0$--estimate in last section, we know that there exists a uniform constant $\hat{C}$ depending only on $X$, $t_0$ and $\omega_{TKS}$, such that for all $t\in[t_0,1]$,
\begin{equation}
\label{eqn:32}
\operatorname{OSC}_M\lr{\varphi_t-\varphi_1}\leq \hat{C}\lr{1+I_{\omega_{TKS}}\lr{\varphi_t-\varphi_1}}.
\end{equation}

In particular, let $\varphi'=\varphi_t-\varphi_1$, then we get that $$\widetilde{F}_{\omega_{TKS},\eta}\lr{\varphi'} \geq CI_{\omega_{TKS}}^{\frac{1}{4n+5}} \lr{\varphi'} - C',$$ i.e.
\begin{equation}
\label{eqn:30}
  \widetilde{F}_{\omega_{g},\eta}\lr{\varphi_t}-\widetilde{F}_{\omega_{g},\eta}\lr{\varphi_1} \geq CI_{\omega_{TKS}}^{\frac{1}{4n+5}} \lr{\varphi'} - C'.
\end{equation}
Further more, we have that
\begin{equation*}
\widetilde{F}_{\omega_{g},\eta}\lr{\varphi_t}- \widetilde{F}_{\omega_{g},\eta}\lr{\varphi_1} \leq\lr{1-t}\lr{\lr{\widetilde{I}_{\omega_g} \lr{\varphi_1} -\widetilde{J}_{\omega_g} \lr{\varphi_1}}-\lr{\widetilde{I}_{\omega_g} \lr{\varphi_t} -\widetilde{J}_{\omega_g} \lr{\varphi_t}}}
\end{equation*}

By direct calculation, we get that
\begin{equation*}
\begin{split}
  &\repl{=}\lr{\widetilde{I}_{\omega_g} \lr{\varphi_1} -\widetilde{J}_{\omega_g} \lr{\varphi_1}}-\lr{\widetilde{I}_{\omega_g} \lr{\varphi_t} -\widetilde{J}_{\omega_g} \lr{\varphi_t}}\\
  &=\frac{1}{V}\int_M\varphi_t\lr{e^{\theta_X\lr{\omega_{\varphi_t}}}\omega_{\varphi_t}^n -e^{\theta_X\lr{\omega_{\varphi_1}}}\omega_{\varphi_1}^n} + \widetilde{J}_{\omega_{\varphi_1}}\lr{\varphi_t-\varphi_1}
\end{split}
\end{equation*}
and
\begin{align*}
  &\repl{=}\frac{1}{V}\int_M\varphi_t\lr{e^{\theta_X\lr{\omega_{\varphi_t}}}\omega_{\varphi_t}^n -e^{\theta_X(\omega_{\varphi_1})}\omega_{\varphi_1}^n}\\
  &=\frac{1}{V}\int_0^1\int_M \lr{\varphi_t-\varphi_1}\lr{\laplace_{s,t} +X}\lr{\varphi_t}e^{\theta_X\lr{s\omega_{\varphi_t}+(1-s)\omega_{\varphi_1}}} \omega_{\varphi_1+s\lr{\varphi_t-\varphi_1}}^n\wedge \dd s\\
  &\leq C \operatorname{OSC}_M\lr{\varphi_t-\varphi_1}
\end{align*}
with the notation that $\laplace_{s,t}=\laplace_{\omega_{\varphi_1+s\lr{\varphi_t-\varphi_1}}}$, for some uniform constant $C$ depending only on $X$ and $\omega_{TKS}$. So we get that
\begin{equation*}
  \widetilde{F}_{\omega_{TKS},\eta}\lr{\varphi_t-\varphi_1}\leq C''\lr{1-t}\lr{1+I_{\omega_{TKS}}\lr{\varphi_t-\varphi_1}}.
\end{equation*}
Combining with \eqref{eqn:30}, we get the inequality as follow with $\alpha=\frac{4n+4}{4n+5}$:
\begin{equation*}
  C\frac{I_{\omega_{TKS}} \lr{\varphi_t-\varphi_1}}{1+I_{\omega_{TKS}}^{\alpha} \lr{\varphi_t-\varphi_1}} - C'\leq CI_{\omega_{TKS}}^{1-\alpha} \lr{\varphi_t-\varphi_1} - C'\leq C''\lr{1-t}\lr{1+I_{\omega_{TKS}} \lr{\varphi_t-\varphi_1}},
\end{equation*}
i.e.
\begin{equation}
\label{eqn:35}
  \frac{I_{\omega_{TKS}} \lr{\varphi_t-\varphi_1}}{1+I_{\omega_{TKS}}^{\alpha} \lr{\varphi_t-\varphi_1}} [C-C''\lr{1-t}\lr{1+I_{\omega_{TKS}}^{\alpha} \lr{\varphi_t-\varphi_1}}]\leq C''\lr{1-t}+C',
\end{equation}
for some uniform constants $C$, $C'$ and $C''$ depending only on $X$, $\eta$ and the geometry of $\lr{M,\omega_{TKS}}$. We can suppose that there exists $t'\in[t_0,1]$ such that
\begin{equation}
\label{eqn:34}
C''\lr{1-t'}\lr{1+I_{\omega_{TKS}}^{\alpha} \lr{\varphi_t-\varphi_1}}= \frac{C}{2},
\end{equation}
Then \eqref{eqn:35} and \eqref{eqn:34} implies that $I_{\omega_{TKS}}\lr{\varphi_{t'}-\varphi_1}\leq \widetilde{C}$ and $\lr{1-t'}\geq \hat{C}$, where $\widetilde{C}$ and $\hat{C}$ are two uniform constants depending only on $X$, $\eta$ and the geometry of $\lr{M,\omega_{TKS}}$, and also implies the inequality required by choosing $t=t'$ and combining with \eqref{eqn:26} and \eqref{eqn:32}.

If \eqref{eqn:34} is not true, we must have that $$C''\lr{1-t_0}\lr{1+I_{\omega_{TKS}}^{\alpha} \lr{\varphi_{t_0}-\varphi_1}}<\frac{C}{2},$$ and it follows that $$I_{\omega_{TKS}}\lr{\varphi_{t_0}-\varphi_1}<\widetilde{C}$$ for some uniform constant $\widetilde{C}$ depending only on $X$, $\eta$ and the geometry of $\lr{M,\omega_{TKS}}$. Then choosing $t=t_0$, \eqref{eqn:26} and \eqref{eqn:32} implies the inequality required.

\textbf{Case 2:} For all $t$ satisfying $\lr{\lr{1-t}\beta}^{\frac{1}{4 \lr{n+1}}} \leq \frac{\sigma}{32C_0}$, we have
\begin{equation*}
  \lr{1-t}\left|\left|\varphi_t\right|\right|_{C^{0}\lr{M}}\lr{\lr{1-t}\beta}^{\frac{1}{4\lr{n+1}}} < \frac{\sigma}{32C_0}
\end{equation*}
where $C_0$ is the constant in \eqref{eqn:25}. With $1-t_0=\frac{1}{\beta}\lr{\frac{\sigma}{32C_0}}^{4n+4}$, we have $$\left|\left|\varphi_{t_0}\right|\right|_{C^{0}\lr{M}}<\beta \lr{\frac{\sigma}{32C_0}}^{-4n-4}.$$ Similarly to the discussion to getting \eqref{eq:1}, we get that
\begin{equation}
\label{eqn:36}
\begin{split}
\operatorname{OSC}_M \lr{\varphi_{t_0} -\varphi_1} &\leq \frac{1}{2}+12e^\beta\lr{1-t_0}\left|\left|\varphi_t\right|\right|_{C^0\lr{M}}\\
&\leq C',
\end{split}
\end{equation}
where $C'$ is a uniform constant depending only on $X$, $\eta$ and $\omega_{TKS}$.
Then the inequality required is a consequence of \eqref{eqn:26} and \eqref{eqn:36} while $t=t_0$.
\end{proof}
\begin{zj}
For the theorem above, we have that
\begin{enumerate}[1).]
  \item $\widetilde{\mu}_{\omega_{TKS},\eta}$ is proper is just an easy consequence of the theorem above and \eqref{eqn:3},
  \item The inequality in Theorem \ref{theorem:1} is a consequence of the co--cycle proposition, \eqref{eqn:5} and the theorem above.
\end{enumerate}
\end{zj}

As a consequence of Theorem \ref{theorem:1} and \ref{eqn:5}, we have the following corollary,

\begin{tl}
\label{theorem:2}
  The following conditions are equivalent:
  \begin{enumerate}[(1).]
    \item $0<\beta<R\lr{X}$
    \item $\exists \text{ } \eta\in \lr{1-\beta}\mathscr{K}^0_X\lr{\omega_0}$, there is a twisted K\"{a}hler--Ricci soliton in $\mathscr{K}_X(\omega_0)$ such that
        \begin{equation*}
          Ric(\omega)=\beta\omega+\eta+L_{X}\omega,
        \end{equation*}
    \item $\forall \text{ } \eta\in \lr{1-\beta}\mathscr{K}^0_X\lr{\omega_0}$, there exists a twisted K\"{a}hler--Ricci soliton such that
        \begin{equation*}
          Ric\lr{\omega}=\beta\omega+\eta+L_{X}\omega,
        \end{equation*}
    \item $\forall \eta \in \lr{1-\beta}\mathscr{K}^0_X\lr{\omega_0}$, the twisted Mabuchi $K$--energy $\widetilde{\mu}_{\omega,\eta}$ is $\widetilde{J}$--proper, or the Morse--Trudinger inequality is hold.
  \end{enumerate}
\end{tl}

\begin{proof}
We will do this as follow:
  
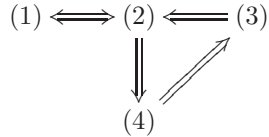
\[\xymatrix{
  (1) & (2) \ar@{=>}[d] \ar@{<=}[r] \ar@{<=>}[l] &    (3)    \\
  &(4) \ar@{=>}[ur]     &                }\]
  (2)$\Rightarrow$(1), and (3)$\Rightarrow$(2) are obvious.
  \begin{itemize}
    \item (1)$\Rightarrow$(2) By the definition of $R\lr{X}$, there exists $\hat{\omega}\in \mathscr{K}_X(\omega_0)$, such that $$\hat{\eta}=Ric(\hat{\omega})-\beta \hat{\omega}-L_X\hat{\omega},$$ and it is easy to check that $\hat{\eta}\in\lr{1-\beta} \mathscr{K}^0_X\lr{\omega_0}$.
    \item (4)$\Rightarrow$(3) This is an easy consequence of Theorem \ref{theorem:1}.
    \item (2)$\Rightarrow$(4) Assuming that $$\hat{\eta}=Ric(\hat{\omega})-\beta \hat{\omega}-L_X\hat{\omega}$$ is a $(1,1)$--form in $(1-\beta)\mathscr{K}^0_X(\omega_0)$ for $\hat{\omega}\in\mathscr{K}_X(\omega_0)$, it is easy to check that $$\int_{M}\hat{\eta} \wedge \hat{\omega}^{n-1}=1-\beta>0,$$ which implies that $\hat{\eta}$ is strictly positive at one point. Theorem \ref{theorem:3} implies that $\widetilde{\mu}_{\widetilde{\omega},\hat{\eta}}$ is proper, and (4) is just a consequence of the proper proposition and Lemma \ref{lemma:f}.
  \end{itemize}
  Then the corollary is proved.
\end{proof}

\begin{zj}
Corollary \ref{corallary:1} is a consequence of the corollary above.
\end{zj}

\section{The existence of conical K\"{a}hler-Ricci soliton}
\label{section:5}
Before proving Theorem \ref{theorem:0.4}, we first prove the existence of conical K\"{a}hler--Ricci soliton under the properness of the log modified Mabuchi $K$--energy or Ding--functional.


In order to make our computations make sense, we will add an assumption to the divisor considered, or more precisely the line bundle of the divisor. Let $D\in|L|$ be a divisor, and $L$ is an ample line bundle over $M$ such that
\begin{equation*}
c_1(L)=\lambda c_1(M).
\end{equation*}
$s$ is a section of $L$ determining $D$ and $H$ is a Hermitian metric of $L$ with curvature $\lambda \omega_0$. Furthermore, we assume that $s$ satisfies that for some constant $C$,
\begin{equation}
\label{eqn:S23}
  |X(\log(|s|_H^2))|<C.
\end{equation}
According to \cite{li2012conical}, we know that $\lambda$ is a rational number in $(0,1)$.

\begin{yl}
  If $|X(\log(|s|_H^2))|<C$, then,
    \begin{equation*}
    \label{eqn:S1}
    \operatorname{Im}X(|s|_H^2)=0.
    \end{equation*}
\end{yl}

This lemma is an application of a Stokes' formula by C.T.Simpson (cf. Lemma 5.2 of \cite{simpson1988constructing}). We give the details of the proof for the readers' convenience.

\begin{proof}
  Let $\eta=\operatorname{Im}X(\log|s|_H^2) \overline{\partial} \operatorname{Im}X (\log|s|_H^2) \wedge \omega_0^{n-1}$. Obviously, $\int_{M\setminus D}|\eta|^2_{\omega_0}\omega_0^n$ is bounded. Furthermore, $M\setminus D$ satisfies the conditions in Lemma 5.2 of \cite{simpson1988constructing}. Since $\partial \overline{\partial} \operatorname{Im}X(\log|s|_H^2) \equiv 0$ on $M\setminus D$, the Stokes' formula to $\eta$ implies that
  \begin{equation}
    \int_{M\setminus D}|\bar{\partial} \operatorname{Im}X(\log|s|_H^2)|_{\omega_0} \omega_0^n =\int_{M\setminus D}\sqrt{-1}d\eta =0.
  \end{equation}
  So $\operatorname{Im}X(\log|s|_H^2)$ is a constant on $M\setminus D$, i.e.
  \begin{equation}
    \operatorname{Im}X(|s|_H^2)=C|s|_H^2, \text{ on }M\setminus D.
  \end{equation}
  However, $|s|_H^2=0$ on $D$ and $|s|_H^2>0$ in $M\setminus D$, so $|s|_H^2$ must achieve its maximum in the inner of $M\setminus D$. It is easy to see that $C=0$, i.e. $$\operatorname{Im}X(|s|_H^2)=0,$$
  on $M$.
\end{proof}

For $\nu\in [0,1]$, let $h_\omega$ be the Ricci potential define by $$Ric\lr{\omega}=\gamma\lr{\lambda,\nu}\omega+\lambda\nu\omega_0+ \im\partial\overline{\partial}h_\omega,$$ normalized by $\int_M e^{h_\omega}\omega^n=\int_M\omega_0^n$ and $\gamma\lr{\lambda,\nu}=1-\lambda\nu$. For convenience, we define $H$ by multiplying a constant such that $$\int_M \frac{e^{h_{\omega_0}}}{|s|_{H}^{2\nu}} \omega_0^n = \int_M\omega_0^n.$$ We define the log modified Mabuchi $K$--energy and Ding--functional on $\mathscr{H}_X(M,\omega_0)$ as follow:
\begin{equation*}
  \widetilde{\mu}_{\omega_0,\nu D}\lr{\varphi}=\widetilde{\mu}_{\omega_0,\nu\lambda\omega_0}\lr{\varphi} + \frac{\nu}{V}\int_M\log\lr{|s|_H^2}\lr{e^{\theta_X\lr{\omega_\varphi}}\omega_\varphi^n-e^{\theta_X}\omega_0^n},
\end{equation*}
and
\begin{equation*}
  \widetilde{F}_{\omega_0,\nu D}\lr{\varphi}= \hat{F}_{\omega_0}\lr{\varphi}-\frac{1}{\gamma\lr{\lambda,\nu}}\log\lr{\frac{1}{V}\int_M\frac{e^{h_0-\gamma\lr{\lambda, \nu} \varphi}}{|s|^{2\nu}_H}\omega_0^n},
\end{equation*}
with the notation that $h_0=h_{\omega_0}$.
We can check that both of them are well-defined and satisfy the cocycle property. Similarly to Lemma \ref{lemma:2}, we have the following lemma for $\widetilde{\mu}_{\omega_0,\nu D}$ and $\widetilde{F}_{\omega_0,\nu D}$:
\begin{yl}
  \begin{equation}
  \label{eqn:D1}
    \widetilde{\mu}_{\omega_0, \nu D}\geq \gamma\lr{\lambda, \nu} \widetilde{F}_{\omega_0,\nu D}+ \frac{1}{V} \int_M \lr{h_0-\theta_X-\nu\log|s|_H^2}e^{\theta_X}\omega_0^n.
  \end{equation}
\end{yl}

The Poincar\`{e}--Lelong formula implies that $[D]=\lambda\omega_0+\im\partial\overline{\partial}\log|s|_H^2$. The positive $\lr{1,1}$--form $\eta_\varepsilon=\lambda\omega_0+ \im\partial\overline{\partial} \log\lr{|s|_H^2+\varepsilon^2}$ converges to $[D]$ as a current while $\varepsilon\rightarrow 0$. With this notation, we have that,
\begin{yl}
  If $\widetilde{\mu}_{\omega_0,\nu D}$ or $\widetilde{F}_{\omega_0, \nu D}$ is proper on the function space $\mathscr{H}_X\lr{M,\omega_0}$, then $\widetilde{\mu}_{\omega_0, \nu \eta_\varepsilon}$ is proper uniformly for all $\varepsilon\in(0,1]$.
\end{yl}
\begin{proof}
  Since \eqref{eqn:D1}, we just consider when $\widetilde{\mu}_{\omega_0, \nu D}$ is proper. In order to get the properness of $\widetilde{\mu}_{\omega_0, \nu \eta_\varepsilon}$, we do computation as follow, for all $\varphi\in \mathscr{H}_X\lr{M,\omega_0}$:
  \begin{align*}
      \repl{=}&\widetilde{\mu}_{ \omega_0,\nu \eta_\varepsilon} \lr{\varphi}-\widetilde{\mu}_{\omega_0, \nu D}\lr{\varphi}\\
      =&\frac{\nu}{V}\int_M \log\frac{|s|^2_H+\varepsilon^2}{|s|_H^2}\lr{e^{\theta_X\lr{\omega_\varphi}}\omega_\varphi^n -e^{\theta_{X}\lr{\omega_0}}\omega_0^n}\\
      \geq&\frac{\nu}{V}\int_M \log \frac{|s|_H^2}{|s|^2_H+\varepsilon^2}e^{\theta_X}\omega^n_0\geq-C
,
  \end{align*}
  where $C$ is independent of $\varphi$. The uniform properness of $\widetilde{\mu}_{ \omega_0,\nu \eta_\varepsilon}$ is a consequence of the inequality.
\end{proof}

Before proving Theorem \ref{theorem:existence}, we recall some facts in \cite{guenancia2013conic} which will be useful while getting the Laplace estimate. We denote
$$\omega_\varepsilon=\omega_0+k\im \partial\overline{\partial}\chi \lr{\varepsilon^2 + |s|_H^2},$$
where
$$\chi\lr{\varepsilon^2+t}= \frac{1}{1-\nu}\int_0^t \frac{(\varepsilon^2+r)^{1-\nu} -\varepsilon^{2-2\nu} }{r}\dd r, $$
 and $k$ is a sufficiently small number such that $\omega_\varepsilon$ is a K\"{a}hler form for each $\varepsilon\in\lr{0,1}$. $\omega_\varepsilon\rightarrow \omega^*$ in the sense of current globally on $M$ and in $C_{\text{loc}}^\infty$ topology outside $D$, where $$\omega^*=\omega_0+k\im\partial\overline{\partial} |s|_H^{2-2\nu}$$ is a conical K\"{a}hler metric with cone angle $2\pi\lr{1-\nu}$ along $D$. From \cite{guenancia2013conic}, we know that the function $\chi\lr{\varepsilon^2+t}$ is smooth for any $\varepsilon>0$, and there exist constants $C>0$ and $\gamma>0$ independent of $\varepsilon$ such that
\begin{equation*}
  0\leq \chi\lr{\varepsilon^2+t}\leq C,
\end{equation*}
provided that $t$ belongs to a bounded interval and
\begin{equation}
\label{eqn:comparison of w w0}
  \omega_\varepsilon\geq \gamma\omega_0.
\end{equation}

Now we give the proof of Theorem \ref{theorem:existence} as follow:

\begin{proof}
  In order to solve \eqref{eqn:D2}, we will follow the method of \cite{liu2014conical} to get a uniform global $C^{0,\tau}$--estimate, Laplace estimate and local $C^{k,\alpha}$--estimate ($k\in\mathbb{Z}^+$) for the following equation:
  \begin{equation}
    \label{eqn:D3}
    Ric\lr{\omega_{\phi_{\varepsilon,t}}}=t\gamma\lr{\lambda, \nu} \omega_{\phi_{\varepsilon,t}} + \lr{1-t}\gamma\lr{\lambda,\nu}\omega_0+ \nu \eta_\varepsilon + L_X \omega_{\phi_{\varepsilon,t}},
  \end{equation}
  where $\eta_\varepsilon$ is given above and $t\in[\frac{1}{2},1]$, $\varepsilon\in(0,1)$. With $\omega_{\phi_{\varepsilon,t}} =\omega_\varepsilon+\im \partial\overline{\partial } \varphi_{\varepsilon,t}$ i.e. $$\phi_{\varepsilon,t} =\varphi_{\varepsilon,t}+k\chi\lr{|s|_H^2+\varepsilon^2}$$ the scalar version of \eqref{eqn:D3} is
  \begin{equation}
    \label{eqn:D3scalar}
    \frac{\lr{\omega_\varepsilon+\im \partial\overline{\partial } \varphi_{\varepsilon,t}}^n}{\omega_0^n} = \frac{e^{h_0 - t \gamma \lr{\lambda,\nu } \lr{\varphi_{\varepsilon,t}+k\chi} -\theta_X -X\lr{k\chi+\varphi_{\varepsilon,t}}}}{\lr{|s|_H^2+\varepsilon^2}^\nu}.
  \end{equation}

  \textbf{Step 1: $C^0$ and H\"older estimate}

  To get uniform $C^0$--estimate for the K\"{a}hler potential $k\chi+\varphi_{\varepsilon,t}$, we will follow the proof of Proposition \ref{mt:2.6}. We just need the uniform up bound of $\widetilde{\mu}_{\omega_0,\nu\eta_\varepsilon}$, the uniform properness of Mabuchi $K$--energy $\widetilde{\mu}_{ \omega_0,\nu \eta_\varepsilon}(\phi_{\varepsilon,t})$ which is true according to the lemma above and uniform lower bound of the Green function for the operator $$\laplace_{\omega_{\phi_{\varepsilon,t}}}+X$$ for any $t\in[\frac{1}{2},1]$ and $\varepsilon\in(0,1]$. The uniform up bound of $\widetilde{\mu}_{ \omega_0,\nu \eta_\varepsilon}(\phi_{\varepsilon,t})$ is given by
  \begin{equation*}
  \begin{split}
    \widetilde{\mu}_{ \omega_0,\nu \eta_\varepsilon}(\phi_{\varepsilon,t})&\leq \frac{1}{V}\int_M(h_0 -\nu\log(|s|_H^2+ \varepsilon^2)-\theta_X) \omega_0^n\\
    &\leq\frac{1}{V}\int_M(h_0 -\nu \log(|s|_H^2+ \varepsilon^2)-\theta_X) \omega_0^n.
  \end{split}
  \end{equation*}
  According to Lemma \ref{lemma:6}, we just find two uniform constants $\widetilde{\lambda}$ and $\widetilde{k}$, such that
  \begin{equation*}
    \begin{split}
      Ric\lr{\omega_{\phi_{\varepsilon,t}}}-L_X\omega_{\phi_{\varepsilon,t}}&\geq \widetilde{\lambda} \omega_{\phi_{\varepsilon,t}},\\
      \laplace_{\omega_{\phi_{\varepsilon,t}}}\theta_X\lr{\omega_{\phi_{\varepsilon,t}}} & \leq \widetilde{k}.
    \end{split}
  \end{equation*}
  Since \eqref{eqn:D3} and $\eta_\varepsilon$ and $\omega_0$ are positive, we just choose $$\widetilde{\lambda}= \frac{1}{2} \gamma\lr{\lambda,\nu}. $$ Direct computation shows that
  \begin{equation}
  \label{eqn:D4}
    \begin{split}
      \repl{=}&\laplace_{\omega_{\phi_{\varepsilon,t}}}\theta_X\lr{\omega_{\phi_{\varepsilon,t}}}\\
      =& -\gamma\lr{\lambda,\nu}\theta_X\lr{\omega_{\phi_{\varepsilon,t}}} - \lambda\nu \theta_X -X\lr{\theta_X\lr{\omega_{\phi_{\varepsilon,t}}}}\\
      & - \nu X\lr{\log(|s|_H^2+\varepsilon^2)}+\gamma\lr{\lambda,\nu}\lr{1-t} X\lr{\phi_{\varepsilon,t}} + C_{\varepsilon,t},
    \end{split}
  \end{equation}
  where $C_{\varepsilon,t}$ is a constant. Lemma \ref{lemma:3} and the equation \eqref{eqn:S1} implies that there exists positive constant $C'$ which is just dependent on $\omega_0$, $\lambda$, $\nu$ and $X$, such that
  \begin{equation*}
    |X\lr{\log(|s|_H^2+\varepsilon^2)}|,|X\lr{\phi_{\varepsilon,t}}|\leq C'.
  \end{equation*}
  Applying the maximal principle to the equation \eqref{eqn:D4}, we get that
  \begin{equation*}
    C_{\varepsilon,t}\leq \gamma\lr{\lambda,\nu}\mathop{\operatorname{OSC}}_M |\theta_X \lr{\omega_{\phi_{\varepsilon , t}}}| + C''
  \end{equation*}
  where $C''$ is a uniform constant dependent on $\lambda$, $\nu$, $C'$ and the $C^0$--norm of $\theta_X$. Now we rewrite the equation \eqref{eqn:D4} as follow
  \begin{equation*}
    \begin{split}
      \repl{=}&\laplace_{\omega_{\phi_{\varepsilon,t}}}\theta_X \lr{\omega_{\phi_{\varepsilon,t}}}\\
      =& -\gamma\lr{\lambda,\nu}\theta_X\lr{\omega_{\phi_{\varepsilon,t}}} - \lambda\nu \theta_X -\left|\left|X\right|\right|^2_{\omega_{\phi_{\epsilon,t}}} - \nu X\lr{\log(|s|_H^2+\varepsilon^2)}\\
      &+\gamma\lr{\lambda,\nu}\lr{1-t} X\lr{\phi_{\varepsilon,t}} + C_{\varepsilon,t}\\
      \leq &2\gamma\lr{\lambda,\nu} \mathop{\operatorname{OSC}}_M|\theta_X| +2\gamma\lr{\lambda,\nu}|X\lr{\phi_{\varepsilon,t}}|- \lambda\nu \theta_X\\
      & - \nu X\lr{\log(|s|_H^2+\varepsilon^2)}+\gamma\lr{\lambda,\nu}\lr{1-t} X\lr{\phi_{\varepsilon,t}}+C''\\
      \leq & C ,
    \end{split}
  \end{equation*}
  where $C$ is a uniform constant independent on $\varepsilon$ and $t$, so we just choose $$\widetilde{k}=C,$$ which implies the uniform $C^0$--estimate for the K\"{a}hler potential $\phi_{\varepsilon,t}$. In particular, the $L^p$ norm of the right--hand side of \eqref{eqn:D3scalar} is uniformly bounded for some fix $p>1$. The uniform global H\"older estimate for $\phi_{\varepsilon,t}$ is just an easy consequence of Ko{\l}odziej's work \cite{kolodziej2008holder}.

  \textbf{Step 2: uniform Laplace estimate}

  We want to prove that there exist a uniform positive constant $A$ such that for any $\varepsilon\in (0,1)$ , $\omega_{\phi_{\varepsilon}}$ satisfies that
  \begin{equation}
  \label{eqn:laplace5}
    \frac{1}{A}\omega_\varepsilon\leq \omega_{\phi_{\varepsilon}}\leq A\omega_\varepsilon
  \end{equation}
  with the notation that $\omega_{\phi_\varepsilon}=\omega_{\phi_{\varepsilon,1}}$. Rewrite the equation \eqref{eqn:D3scalar} at $t=1$ as follow:
  \begin{equation*}
    \log\frac{\lr{\omega_\varepsilon+\im\partial\overline{\partial} \varphi_\varepsilon}^n}{\omega_\varepsilon^n} = h_\varepsilon - \gamma\lr{\lambda,\nu}\lr{\varphi_\varepsilon+k\chi}-\theta_X - X\lr{k\chi+\varphi_\varepsilon}
  \end{equation*}
  where
  \begin{equation}
  \label{1}
  h_\varepsilon=h_0- \log[\frac{\omega_\varepsilon^n}{\omega_0^n} \lr{|s|_H^2+\varepsilon^2}^\nu],
  \end{equation}
  and we denote $F_\varepsilon=\log[\frac{\omega_\varepsilon^n}{\omega_0^n} \lr{|s|_H^2+\varepsilon^2}^\nu]$ which is bounded independent of $\varepsilon$ according to \cite{guenancia2013conic}. We also have the notation as follow
   \begin{equation*}
     \laplace'=\laplace_{\omega_{\phi_{\varepsilon}}}, \laplace=\laplace_{\omega_{\varepsilon}}.
   \end{equation*}
  Given $p\in M$, we choose a locally geodesic holomorphic basis $\{w_i\}$, such that
  \begin{equation*}
    g_{i\overline{j}}(p)=\delta_{ij} \text{, and }g_{\phi_\varepsilon i\overline{j}}(p)=\delta_{ij}+\delta_{ij} \varphi_{\varepsilon i\overline{j}},
  \end{equation*}
  where $\varphi_\varepsilon=\varphi_{\varepsilon,1}$, and $g_\varepsilon$, $g_{\phi_{\varepsilon}}$ are the local representation of $\omega_\varepsilon$ and $\omega_{\phi_\varepsilon}$. Direct computation shows that
  \begin{equation}
  \label{eqn:la1}
  \begin{split}
    \repl{=}&\laplace'\lr{\log\operatorname{tr}_{\omega_\varepsilon} \lr{\omega_\varepsilon+\im \partial \overline{\partial} \varphi_{\epsilon}}}\\
    =&\frac{1}{n+\laplace\varphi_{\varepsilon}}[\laplace h_\varepsilon -\gamma\lr{\lambda,\nu} \lr{k\laplace \chi + \laplace \varphi_\varepsilon}-\laplace\theta_X\lr{\omega_{\phi_\varepsilon}}\\
    \repl{=}&+ \frac{\varphi_{\varepsilon j\overline{i}\overline{l}} \varphi_{\varepsilon\overline{j}il}}{(1+\varphi_{\varepsilon i\overline{i}})(1+\varphi_{\varepsilon j\overline{j}})} +\frac{1}{1+\varphi_{\varepsilon i \overline{i}}} R_{\omega_\varepsilon i \overline{i} k \overline{k}}- R_{\omega_\varepsilon i \overline{i} k \overline{k}} \\
    \repl{=}& + \frac{\varphi_{\varepsilon i \overline{i}}}{1+ \varphi_{\varepsilon k \overline{k}}} R_{\omega_\varepsilon i\overline{i}k\overline{k}}] -\frac{1}{\lr{n+\laplace\varphi_{\varepsilon}}^2} \frac{\varphi_{\varepsilon i\overline{i}k}\varphi_{\varepsilon \overline{i}i\overline{k}}}{1+ \varphi_{\varepsilon k \overline{k}}}.
  \end{split}
  \end{equation}
  By \cite{tian2000book},
  \begin{equation}
  \label{eqn:la2}
    \frac{1}{n+\laplace\varphi_\varepsilon}\frac{\varphi_{\varepsilon j\overline{i}\overline{l}} \varphi_{\varepsilon\overline{j}il}}{(1+\varphi_{\varepsilon i\overline{i}})(1+\varphi_{\varepsilon j\overline{j}})} - \frac{1}{\lr{n+\laplace\varphi_{\varepsilon}}^2} \frac{\varphi_{\varepsilon i\overline{i}k}\varphi_{\varepsilon \overline{i}i\overline{k}}}{1+ \varphi_{\varepsilon k \overline{k}}} \geq 0.
  \end{equation}
  On the other hand, we have
  \begin{equation}
  \label{eqn:la3}
  \begin{split}
    \repl{=}&\frac{1}{1+\varphi_{\varepsilon i\overline{i}}} R_{\omega_\varepsilon i\overline{i}k\overline{k}}- R_{\omega_\varepsilon i \overline{i} k \overline{k}} + \frac{\varphi_{\varepsilon i \overline{i}}}{1+ \varphi_{\varepsilon k \overline{k}}} R_{\omega_\varepsilon i\overline{i}k\overline{k}}\\
    =&R_{\omega_\varepsilon i\overline{i}k\overline{k}}\lr{ \frac{1}{1+\varphi_{\varepsilon i\overline{i}}}- 1+ \frac{\varphi_{\varepsilon i \overline{i}}}{1+ \varphi_{\varepsilon k \overline{k}}}}\\
    =&\frac{1}{2}R_{\omega_\varepsilon i\overline{i}k\overline{k}}\lr{\frac{1+\varphi_{\varepsilon i\overline{i}}}{1 + \varphi_{\varepsilon k\overline{k}}}+ \frac{1+\varphi_{\varepsilon k \overline{k}}}{1+\varphi_{\varepsilon i\overline{i}}}-2},
  \end{split}
  \end{equation}
  and
  \begin{equation}
  \label{eqn:la31}
    n=\operatorname{tr}_{\omega_\varepsilon}\omega_0+k\laplace \chi\lr{|s|_H^2+\varepsilon^2}\geq k\laplace \chi\lr{|s|_H^2+\varepsilon^2}.
  \end{equation}
  Further more, by direct computation, the following inequality is true
  \begin{equation}
  \label{eqn:la4}
    \laplace\theta_X\lr{\omega_{\phi_\varepsilon}}\leq \lr{n+\laplace\varphi_\varepsilon}\max\{\sup X^j_{,j}, 0\} + X^j \varphi_{\varepsilon j \overline{i}i} .
  \end{equation}
  Combining the equation \eqref{1}, \eqref{eqn:la1}, \eqref{eqn:la2}, \eqref{eqn:la3},\eqref{eqn:la31} and \eqref{eqn:la4}, we get that
  \begin{equation}
  \label{eqn:laplace1}
    \begin{split}
      \repl{\geq}&\laplace'\lr{\log\operatorname{tr}_{\omega_\varepsilon}\lr{\omega_\varepsilon+\im \partial \overline{\partial} \varphi_{\epsilon}}}\\
      \geq& \frac{\laplace h_0}{n+\laplace \varphi_\varepsilon}-\frac{\laplace F_\varepsilon}{n+\laplace \varphi_\varepsilon}- \gamma\lr{\lambda,\nu} -\max\{\sup X^j_{,j}, 0\} -\frac{X^j \varphi_{\varepsilon j \overline{i}i}}{n+\laplace \varphi_\varepsilon} \\
      \repl{=}&+ \frac{1}{2\lr{n+\laplace \varphi_\varepsilon}}R_{\omega_\varepsilon i\overline{i}k\overline{k}}\lr{\frac{1+\varphi_{\varepsilon i\overline{i}}}{1 + \varphi_{\varepsilon k\overline{k}}}+ \frac{1+\varphi_{\varepsilon k \overline{k}}}{1+\varphi_{\varepsilon i\overline{i}}}-2}.
    \end{split}
  \end{equation}
  There exists a uniform positive constant $C_1$ such that
  \begin{equation*}
    \im \partial\overline{\partial} h_0\geq -C_1\omega_0.
  \end{equation*}
  Combining with \eqref{eqn:comparison of w w0}, we have that
  \begin{equation}
  \label{2}
    -C_1n\gamma^{-1}\leq \laplace h_0 \leq \gamma^{-1}\lr{nC_1 + \laplace_{\omega_0}h_0}.
  \end{equation}
  We denote $\Psi_{\varepsilon,\rho}=\widetilde{C}\chi_{\rho}\lr{|s|_H^2+\varepsilon^2}$, where
  \begin{equation*}
    \chi_{\rho}\lr{|s|_H^2+\varepsilon^2}=\frac{1}{\rho}\int_0^{|s|_H^2}\frac{\lr{\varepsilon^2+r}^\rho - \varepsilon^{2\rho}}{r}\dd r.
  \end{equation*}
  Taking suitable uniform constants $\widetilde{C}$ and $\rho$, \cite{guenancia2013conic} have proved the following inequality
  \begin{equation}
  \label{eqn:laplace2}
  \begin{split}
    \laplace'\Psi_{\varepsilon,\rho}\geq&-\frac{1}{2\lr{n+\laplace \varphi_\varepsilon}}R_{\omega_\varepsilon i\overline{i}k\overline{k}}\lr{\frac{1+\varphi_{\varepsilon i\overline{i}}}{1 + \varphi_{\varepsilon k\overline{k}}}+ \frac{1+\varphi_{\varepsilon k \overline{k}}}{1+\varphi_{\varepsilon i\overline{i}}}-2} - \frac{\laplace F_\varepsilon}{n+\laplace \varphi_\varepsilon} \\
    \repl{=}&-\frac{C_2}{n+\laplace \varphi_\varepsilon}\lr{\frac{1+\varphi_{\varepsilon i\overline{i}}}{1 + \varphi_{\varepsilon k\overline{k}}}+ \frac{1+\varphi_{\varepsilon k \overline{k}}}{1+\varphi_{\varepsilon i\overline{i}}}} -\frac{C_2}{n+\laplace \varphi_\varepsilon} -C_2 \operatorname{tr}_{\omega_{\phi_\varepsilon}}\omega_\varepsilon,
  \end{split}
  \end{equation}
  for some uniform positive constant $C_2$. By \eqref{eqn:laplace1}, \eqref{2} and \eqref{eqn:laplace2}, we get that
  \begin{equation}
    \label{eqn:laplace3}
  \begin{split}
    \repl{=}&\laplace'\lr{\log\lr{n+\laplace \varphi_\varepsilon}+\Psi_{\varepsilon,\rho}}\\
    \geq & -\frac{n^2C_3}{n+\laplace \varphi_\varepsilon} - \frac{C_3}{2} \operatorname{tr}_{\omega_{\phi_\varepsilon}}\omega_\varepsilon -\frac{C_3}{4(n+\laplace \varphi_\varepsilon)} \lr{\frac{1+\varphi_{\varepsilon i\overline{i}}}{1 + \varphi_{\varepsilon k\overline{k}}}+ \frac{1+\varphi_{\varepsilon k \overline{k}}}{1+\varphi_{\varepsilon i\overline{i}}}}\\
    \repl{=}&- \gamma\lr{\lambda,\nu} -\max\{\sup X^j_{,j},0\}- \frac{X^j\varphi_{\varepsilon j\overline{i}i}}{n+\laplace \varphi_\varepsilon}\\
    \geq&-\frac{n^2C_3}{n+\laplace \varphi_\varepsilon} - C_3 \operatorname{tr}_{\omega_{\phi_\varepsilon}}\omega_\varepsilon- \gamma\lr{\lambda,\nu} -\max\{\sup X^j_{,j},0\}- \frac{X^j\varphi_{\varepsilon j\overline{i}i}}{n+\laplace \varphi_\varepsilon}\\
    \geq& -2C_3\operatorname{tr}_{\omega_{\phi_\varepsilon}}\omega_\varepsilon - \gamma\lr{\lambda,\nu} -\max\{\sup X^j_{,j},0\}- \frac{X^j\varphi_{\varepsilon j\overline{i}i}}{n+\laplace \varphi_\varepsilon},
  \end{split}
  \end{equation}
  for some uniform positive constant $C_3$, where we use the inequality $$\operatorname{tr}_{\omega_{\phi_\varepsilon}}  \omega_\varepsilon \cdot \operatorname{tr}_{\omega_\varepsilon} \omega_{\phi_\varepsilon}\geq n^2$$ in the last inequality. Choosing $B=2C_3+1$, we have that
  \begin{equation*}
  \begin{split}
    \repl{=}&\laplace'\lr{\log\lr{n+\laplace \varphi_\varepsilon}+\Psi_{\varepsilon,\rho}-B\varphi_\varepsilon}\\
    \geq& -2C_3\operatorname{tr}_{\omega_{\phi_\varepsilon}}\omega_\varepsilon -B\laplace'\varphi_\varepsilon - \gamma\lr{\lambda,\nu} -\max\{\sup X^j_{,j},0\}- \frac{X^j\varphi_{\varepsilon j\overline{i}i}}{n+\laplace \varphi_\varepsilon}\\
    =&-2C_3\operatorname{tr}_{\omega_{\phi_\varepsilon}}\omega_\varepsilon + B(\operatorname{tr}_{\omega_{\phi_\varepsilon}} \omega_\varepsilon -n) - \gamma\lr{\lambda,\nu} -\max\{\sup X^j_{,j},0\}- \frac{X^j\varphi_{\varepsilon j \overline{i }i}}{n+\laplace \varphi_\varepsilon}\\
    \geq & \operatorname{tr}_{\omega_{\phi_\varepsilon}}\omega_\varepsilon -C_4 - \gamma\lr{\lambda,\nu} -\max\{\sup X^j_{,j},0\}- \frac{X^j\varphi_{\varepsilon j\overline{i}i}}{n+\laplace \varphi_\varepsilon}.
  \end{split}
  \end{equation*}
  We assume that $p\in M$ is the point that $\log\lr{n+\laplace \varphi_\varepsilon}+\Psi_{\varepsilon,\rho}- B\varphi_\varepsilon$ achieve its maximal value. By doing $\frac{\partial}{\partial w^j}$ on the above function, we have that
  \begin{equation}
  \label{eqn:laplace4}
    \varphi_{\varepsilon j \overline{i}i}=\lr{n+\laplace \varphi_\varepsilon}\lr{B\varphi_{\varepsilon j}-\Psi_{\varepsilon, \rho,j}}.
  \end{equation}
  An observation is that
  \begin{equation*}
  \begin{split}
  |X\lr{\varphi_\varepsilon}|&\leq |X\lr{\varphi_\varepsilon+k\chi}|+|X\lr{k\chi}|\leq C_5.
  \end{split}
  \end{equation*}
  Further more, we know that there exists positive number $k'$ independent of $\varepsilon$ such that $\omega_0+k'\im\partial\overline{\partial}\Psi_{\varepsilon,\rho}\geq 0$, Lemma \ref{lemma:3} implies that $X\lr{\Psi_{\varepsilon,\rho}}$ is uniform bounded independent of $\varepsilon$. By multiplying $X^j$ on both sides of \eqref{eqn:laplace4}, we get that
  \begin{equation*}
    \frac{X^j\varphi_{\varepsilon j\overline{i}i}}{n+\laplace \varphi_\varepsilon}\leq C_6
  \end{equation*}
  for some uniform constant $C_6$. On the other hand $X^j_{,j}$ is a uniformly bounded function. The uniform bound of $\Psi_{\varepsilon,\rho}- B\varphi_\varepsilon$ implies that
  \begin{equation*}
    \operatorname{tr}_{\omega_{\phi_\varepsilon}}\omega_\varepsilon\leq C_7
  \end{equation*}
  where $C_7$ is a uniform constant independent of $\varepsilon$. According to \cite{guenancia2013conic}, we know that $F_\varepsilon$ is uniformly bounded w.r.t $\varepsilon$. We have that
  \begin{equation*}
  \begin{split}
    \operatorname{tr}_{\omega_\varepsilon} \omega_{\phi_\varepsilon}&\leq \lr{\operatorname{tr}_{\omega_{\phi_\varepsilon}} \omega_\varepsilon}^{n-1}\frac{\omega_{\phi_\varepsilon}^n}{\omega_\varepsilon}\\
    &\leq C_7 e^{h_0-F_\varepsilon-\gamma\lr{\lambda,\nu} -\theta_X -X\lr{k\chi+\varphi_\varepsilon}}\\
    &\leq C_8.
  \end{split}
  \end{equation*}
  So we prove that there exists a positive uniform constant $A$ independent of $\varepsilon$, such that \eqref{eqn:laplace5} is true.

Similar to \cite{liu2014conical} and \cite{zhu2000kahler}, we can get a uniform local upper bound of $S = |\nabla_{\omega_0} \omega_{ \phi_\varepsilon } |_{\omega_{ \phi_\varepsilon }}$, i.e. $\exists\text{ }C(r)$ independent of $\varepsilon$, such that
\begin{equation*}
S|_{B_p(r)}\leq C(r)
\end{equation*}
where $p\in B_p(r)\subset M\setminus D$. The local uniform higher order estimates is just an easy consequence of elliptic Schauder estimates, i.e. we get a conical metric with some H\"{o}lder function $\phi$ $$\omega_\phi=\omega_{0}+\im\partial\overline{\partial} \phi$$ which is the limit of $\omega_0+\im\partial\overline{\partial} \phi_{\varepsilon}$ globally in the sense of current and locally in the $C^\infty-$topology.

What remains is to check that $\omega_\phi$ satisfies that
\begin{equation}
  \label{CKRS}
  Ric\lr{\omega_\phi}=\gamma\lr{\lambda,\nu} \omega_\phi + \nu [D] +L_{X}\omega_{\phi},
\end{equation}
  globally in the sense of current. Indeed, we have
  \begin{align*}
      &\int_{M} - \im \partial \overline{\partial} \log \frac{\omega_\phi^n|s|_H^{2\nu}}{\omega_0^n}\wedge \zeta\\
      =& \int_{M} - \im \log \frac{\omega_\phi^n|s|_H^{2\nu}}{\omega_0^n} \partial \overline{\partial} \zeta \\
      =&\lim_{\varepsilon\ra 0}\int_{M} - \im \log \frac{\omega_{\phi_\varepsilon}^n \lr{|s|_H^2+\varepsilon^2}^\nu}{\omega_0^n} \partial \overline{\partial}\zeta \\
      =&\lim_{\varepsilon\ra 0}\int_M -\im \lr{h_0 - \gamma \lr{\lambda,\nu } \phi_{\varepsilon} -\theta_X -X\lr{\phi_{\varepsilon}}}\partial\overline{\partial}\zeta\\
      =&\lim_{\varepsilon\ra 0}\int_M -\im\partial\overline{\partial}\lr{h_0 - \gamma \lr{\lambda,\nu } \phi_{\varepsilon} -\theta_X -X\lr{\phi_{\varepsilon}}} \wedge \zeta \\
      =&\lim_{\varepsilon\ra 0}\int_M \lr{\nu\lambda\omega_0 -Ric\lr{\omega_0}+\gamma\lr{\lambda,\nu} \omega_{\phi_\varepsilon} + L_X\omega_{\phi_\varepsilon}}\wedge\zeta\\
      =& \lim_{\varepsilon\ra 0}\int_M \lr{\nu\lambda\omega_0 -Ric\lr{\omega_0}+\gamma\lr{\lambda,\nu} \omega_{\phi_\varepsilon}}\wedge\zeta - \lim_{\varepsilon\ra 0} \int_M \omega_{\phi_\varepsilon}\wedge L_X\zeta\\
      =&  \int_M \lr{\nu\lambda\omega_0 -Ric\lr{\omega_0}+\gamma\lr{\lambda,\nu} \omega_{\phi}}\wedge\zeta - \int_M \omega_{\phi}\wedge L_X\zeta\\
      =&  \int_M \lr{\nu\lambda\omega_0 -Ric\lr{\omega_0}+\gamma\lr{\lambda,\nu} \omega_{\phi}}\wedge\zeta + \int_M L_X\omega_{\phi} \wedge\zeta,
  \end{align*}
  i.e.
  \begin{equation*}
    \int_M Ric\lr{\omega_\phi}\wedge\zeta=\int_M \lr{\gamma\lr{\lambda,\nu}\omega_\phi + \nu [D] +L_X\omega_{\phi}}\wedge \zeta,
  \end{equation*}
  for any test $\lr{n-1,n-1}-$form $\zeta$. So we get that $\omega_\phi$ satisfies the equation \eqref{eqn:D2} globally in the sense of current.
\end{proof}
%
%
%

\section{Some existence results of conical K\"{a}hler--Ricci soliton}
\label{section:6}
In this section we will get some existence results of conical K\"{a}hler--Ricci soliton, according to Theorem \ref{theorem:existence}, i.e. finding some suitable $\lambda$ and $\nu$ such that $\widetilde{\mu}_{\omega_0,\nu D}$ is proper. This is a generalization of \cite{berman2013thermodynamical}, \cite{li2012conical}and \cite{song2012greatest}. We begin this section by proving Theorem \ref{theorem:0.4}.

\begin{zj}
For convenience, we will denote the constant $\tilde{C}$ as the constant $C_2$ appeared in Proposition \ref{prop:1}, which is a positive constant less than $1$.
\end{zj}

\begin{dy}[$\alpha$-invariant]
  \begin{equation*}
  \begin{split}
    \repl{=}&\alpha\lr{[\omega_0],(1-\beta)D}\\
     =&\max\left\{\alpha>0\text{ }|\text{ }\exists\text{ }0<C_\alpha<\infty, \int_Me^{-\alpha\lr{\varphi-\sup \varphi}}\frac{\Omega}{|s|_H^{2-2\beta}}\leq C_\alpha, \forall \varphi\in \mathscr{H}\lr{M,\omega_0}\right\},
  \end{split}
  \end{equation*}
  where $\Omega$ is a smooth volume form.
\end{dy}

According to \cite{berman2013thermodynamical} or \cite{li2012conical}, we have that

\begin{yl}[\cite{berman2013thermodynamical} or \cite{li2012conical}]
If $D\in|L|$, and $c_1(L)=\lambda c_1(M)$, then for $\omega_0\in c_1(M)$,
  \begin{equation*}
    \alpha\lr{[\omega_0],(1-\beta)D}\geq\min\{\lambda \beta,\alpha(\omega_0),\lambda\alpha(L_{|D})\}>0.
  \end{equation*}
\end{yl}
%
%

\begin{mt}
\label{berman}
  For any $\beta$ satisfies that
  \begin{equation*}
    \max\{\frac{1-\lambda}{1-\tilde{C}},0\} <\beta<\min\{\frac{\alpha(\omega_0)}{\tilde{C}}, \frac{\lambda\alpha(L_D|D)}{\tilde{C}},1\},
  \end{equation*}
  the functional $\widetilde{\mu}_{\omega_0, \frac{1-\beta}{\lambda}D}$ is proper on the space $\mathscr{H}_X\lr{M,\omega_0}$, and there exists a conical K\"ahler-Ricci soliton solving equation
  \begin{equation}
  \label{CKRS:beta}
    Ric\lr{\omega_\beta}=\beta\omega_\beta+\frac{1-\beta}{\lambda}[D] + L_X\omega_\beta.
  \end{equation}
\end{mt}

\begin{proof}
  For any $$0<t<\alpha([\omega_0],\frac{1-\beta}{\lambda}D),$$we have
  \begin{equation*}
  \begin{split}
    \log C_t \geq &\log\lr{\frac{1}{V}\int_M e^{-t(\varphi-\sup \varphi)}\frac{e^{\theta_X}\omega_0^n}{|s|_H^{\frac{2-2\beta}{\lambda}}}}\\
    \geq& \frac{1}{V}\int_M [-t(\varphi-\sup \varphi)-\log \frac{|s|_H^{\frac{2-2\beta}{\lambda}}e^{\theta_X+X(\varphi)} \omega_\varphi^n}{e^{\theta_X}\omega_0^n}]e^{\theta_X+X(\varphi)}\omega_\varphi^n\\
    \geq & t\widetilde{I}_{\omega_0}(\varphi)- \frac{1}{V}\int_M \log \frac{|s|_H^{\frac{2-2\beta}{\lambda}}e^{\theta_X+X(\varphi)} \omega_\varphi^n}{e^{\theta_X}\omega_0^n} e^{\theta_X+X(\varphi)}\omega_\varphi^n.
  \end{split}
  \end{equation*}
  By the definition of $\widetilde{\mu}_{\omega_0,D}$, we have
  \begin{equation*}
  \begin{split}
    \repl{=}&\widetilde{\mu}_{\omega_0,\frac{1-\beta}{\lambda}D}(\varphi)\\
    =& -\beta(\widetilde{I}_{\omega_0}(\varphi) - \widetilde{J}_{\omega_0}(\varphi))+\frac{1}{V} \int_M (h_0-\theta_X) (e^{\theta_X}\omega_0^n- e^{\theta_X+X(\varphi)}\omega_\varphi^n) \\
    &-\frac{1-\beta}{\lambda V} \int_M \log|s|_H^2e^{\theta_X}\omega_0^n + \frac{1}{V} \int_M \log\frac{|s|_H^{\frac{2-2\beta}{\lambda}}e^{\theta_X+X(\varphi)}\omega_\varphi^n}{e^{\theta_X}\omega_0^n} e^{\theta_X+X(\varphi)} \omega_\varphi^n\\
    \geq & -\beta(\widetilde{I}_{\omega_0}(\varphi) - \widetilde{J}_{\omega_0}(\varphi)) +t\widetilde{I}_{\omega_0}(\varphi) -C_1\\
    \geq & (t-\beta \tilde{C}) \widetilde{I}_{\omega_0}(\varphi)-C_1.
  \end{split}
  \end{equation*}
  We get that if
  \begin{equation*}
    \max\{\frac{1-\lambda}{1-\tilde{C}},0\}<\beta<\min\{\frac{\alpha(\omega_0)}{\tilde{C}}, \frac{\lambda\alpha(L_D|D)}{\tilde{C}},1\},
  \end{equation*}
  then $\widetilde{\mu}_{\omega_0,\frac{1-\beta}{\lambda}D}$ is proper.

  And the second statement is an easy consequence.
\end{proof}


Next we prove the first part of Theorem \ref{theorem:0.4}.

\begin{proof}
  Assume that $\beta_0$ is a fixed positive number given above such that $\widetilde{\mu}_{\omega_0 ,\frac{1-\beta_0}{\lambda}D}$ is proper. From the definition of $\widetilde{\mu}_{\omega_0, \frac{1-\beta_0}{\lambda}D}$, we get that for some positive constant $C_1,C_2$,
  \begin{equation*}
  \begin{split}
    &\frac{1-\beta_0}{\lambda V}\int_M \log |s|_H^2 (e^{\theta_X+X(\varphi)}\omega_\varphi^n- e^{\theta_X}\omega_0^n)\\
    \geq & (C_1+\beta_0)\lr{\widetilde{I}_{\omega_0}(\varphi) -C_2 - \widetilde{J}_{\omega_0}(\varphi)} -\frac{1}{V}\int_M \log \frac{e^{\theta_X+X(\varphi)}\omega_\varphi^n}{e^{\theta_X}\omega_0^n}e^{\theta_X+X(\varphi)}\omega_\varphi^n.
  \end{split}
  \end{equation*}
  For any $\beta>\beta_0$, let $\xi$ be a positive number letter than $1$, such that
  \begin{equation*}
    0<\frac{\beta-\beta_0-C_1(1-\beta)}{\xi(1-\beta_0)}< \frac{\beta-\beta_0}{1-\beta_0}.
  \end{equation*}
  Further more, we denote $\kappa\in(\frac{\beta-\beta_0-C_1(1-\beta)}{\xi(1-\beta_0)} , \frac{\beta-\beta_0}{1-\beta_0})$. Then we have that

  \begin{equation*}
    \begin{split}
      &\widetilde{\mu}_{\omega_0,\frac{1-\beta_0}{\lambda}D}(\varphi)- \kappa \widetilde{\mu}_{\omega_0,(1-\xi)\omega_0}\lr{\varphi}\\
      \geq&\frac{1}{V}\lr{1-\kappa -\frac{1-\beta}{1-\beta_0}}\int_M\log \frac{e^{\theta_X+X(\varphi)}\omega_\varphi^n}{e^{\theta_X}\omega_0^n}e^{\theta_X+X(\varphi)}\omega_\varphi^n \\
      &+[ \frac{(C_1+\beta_0)(1-\beta)}{1-\beta_0}- \beta +\kappa \xi]\lr{\widetilde{I}_{\omega_0}(\varphi)-\widetilde{J}_{\omega_0}(\varphi)} -C_3 \\
      \geq& C_4\lr{\widetilde{I}_{\omega_0}(\varphi)-\widetilde{J}_{\omega_0}(\varphi)}-C_3.
    \end{split}
  \end{equation*}
  Further more, according to Corollary \ref{theorem:2}, $\widetilde{\mu}_{\omega_0,(1-\xi)\omega_0}$ is proper, so $\widetilde{\mu}_{\omega_0,(1-\beta)D}$ is proper.
\end{proof}


\medskip

While $R(X)=1$, following \cite{chen2012kahler} and \cite{tian2013k}, we consider the limit behavior of $\omega_{\phi_\varepsilon}$ under Gromov--Hausdorff distance for $\beta$ in Theorem \ref{theorem:0.4}, where $\omega_{\phi_\varepsilon}$ is solution of
\begin{equation*}
  Ric(\omega_{\phi_\varepsilon})=\beta \omega_{\phi_\varepsilon}+ \frac{1-\beta}{\lambda}(\lambda\omega_0+\im \partial\overline{\partial}\log (|s|_H^2+\varepsilon^2)) +L_X\omega_{\phi_\epsilon}.
\end{equation*}

Before proving this, we recall the result of \cite{wang2013structure}, i.e. the extended Cheeger--Colding theory in Bakry--Emery geometry.

\begin{thm}[\cite{wang2013structure}]
\label{Theorem:wz}
Let $(M_i,g_i;p_i)$ be a sequence of $n$--dimensional Riemannian manifolds which satisfy
\begin{equation*}
\begin{split}
  &Ric(g_i)+\operatorname{hess}(f_i)\geq -(n-1)^2\Lambda^2 g\\
  &\operatorname{vol}_{g_i}(B_p(1))\geq v>0,\text{ and }|\nabla f_i|_{g_i}\leq A.
\end{split}
\end{equation*}
Then $(M_i,g_i;p_i)$ converge to a metric space $(Y;p_\infty)$ in the pointed Gromov--Hausdorff topology.
\end{thm}

In our case, $f_\varepsilon=\theta_{X}(\omega_{\phi_\varepsilon})$, so we first give a uniform estimate for
\begin{equation*}
|\nabla f_\varepsilon|_{\omega_{\phi_\varepsilon}}=|X|_{\omega_{\phi_\varepsilon}}.
\end{equation*}

\begin{yl}
  $|X|_{\omega_{\phi_\varepsilon}}$ is uniformly bounded with respect to $\varepsilon$.
\end{yl}
\begin{proof}
  We will denote $\theta_X'=\theta_X(\omega_{\phi_\varepsilon})$, and $\laplace$ is the $\overline{\partial}$--Laplace operator associated to $\omega_{\phi_\varepsilon}$. As the computation while getting \eqref{eqn:9}, we have that
  \begin{equation*}
    \laplace\theta_X'= -\beta \theta_X' -(1-\beta) \theta_X - \frac{1-\beta}{\lambda} X(\log(|s|_H^2+\varepsilon^2)) - |X|_{\omega_{\phi_\varepsilon}}^2 + C_1,
  \end{equation*}
  where $C_1$ is a uniform constant. Following \cite{wang2013structure}, the Bochner implies that
  \begin{equation*}
    \begin{split}
      &(\laplace+X)(|X|_{\omega_{\phi_\varepsilon}}^2)\\
      =&|\partial\overline{\partial}\theta_X'|_{\omega_{\phi_\varepsilon}}^2 -\beta |X|_{\omega_{\phi_\varepsilon}}^2 -(1-\beta) |X|_{\omega_{0}}^2 -\frac{1-\beta}{\lambda} \im \partial\overline{\partial}\log (|s|_H^2+\varepsilon^2)(X,\overline{X}).
    \end{split}
  \end{equation*}
  Since $|X\lr{\phi_\varepsilon}|$ is uniformly bounded, we get that
  \begin{equation*}
    |\partial\overline{\partial}\theta_X'|_{\omega_{\phi_\varepsilon}}^2 \geq \frac{(\laplace\theta_X')^2}{n} \geq \frac{(|X|_{\omega_{\phi_\varepsilon}}^2-C_2)^2}{n}.
  \end{equation*}
  Since $X(|s|_H^2)|\leq C_3|s|_H^2$, we have that
  \begin{equation*}
  \begin{split}
    \repl{=}&\im \partial\overline{\partial}\log (|s|_H^2+\varepsilon^2)(X,\bar{X})\\ =&\lr{\frac{\varepsilon^2\langle ds, ds\rangle}{(|s|_H^2+\varepsilon^2)^2}- \frac{\lambda|s|_H^2\omega_0}{|s|_H^2+\varepsilon^2}}(X,\overline{X})\\
    \leq & \frac{\varepsilon^2\langle ds, s\rangle\wedge \langle s, ds\rangle}{|s|^2_H(|s|_H^2+\varepsilon^2)^2}(X,\overline{X})\\
    = & \frac{\varepsilon^2(X(|s|_H^2))^2}{4|s|^2_H(|s|_H^2+\varepsilon^2)^2}\\
    \leq & \frac{C_3^2\varepsilon^2|s|_H^4}{4|s|^2_H(|s|_H^2+\varepsilon^2)^2}\\
    \leq &C_4.
  \end{split}
  \end{equation*}
  Combining the inequalities and equalities above,
  \begin{equation*}
    (\laplace+X)(|X|_{\omega_{\phi_\varepsilon}}^2)\geq \frac{(|X|_{\omega_{\phi_\varepsilon}}^2 -C_2)^2}{n} -\beta |X|_{\omega_{\phi_\varepsilon}}^2- C_5.
  \end{equation*}
  Applying the maximum principle, we get that $|X|_{\omega_{\phi_\varepsilon}}^2$ is uniformly bounded.
\end{proof}

\begin{yl}
\label{lemma:diameter}
  The diameter of $(M,\omega_{\phi_\varepsilon})$ is uniformly bounded. Furthermore, the same thing is true for the volume of geodesic ball $B_p(1)$ with respect to $\omega_{\phi_\varepsilon}$.
\end{yl}
\begin{proof}
  By the result of \cite{mabuchi2003multiplier}, and the condition that
  \begin{equation*}
    Ric(\omega_{\phi_\varepsilon})-L_X\omega_{\phi_\varepsilon}\geq \beta\omega_{\phi_\varepsilon},
  \end{equation*}
  we get that the diameter of $(M,\omega_{\phi_\varepsilon})$ is bounded by $\frac{C_1}{\sqrt{\beta}}$, which is independent of $\varepsilon$.

  It is an easy consequence of the volume comparison theorem for Bakry--Emety Ricci curvature of \cite{wei2007comparison} that the volume of geodesic ball $B_p(1)$ with respect to $\omega_{\phi_\varepsilon}$ is uniformly bounded from above.
\end{proof}

\begin{thm}
\label{theorem:main1}
  The smooth K\"ahler metric $\omega_{\phi_\varepsilon}$ converge to $\omega_\beta$ in the Gromov--Hausdorff topology on $M$ and in the $C^\infty$ topology outside $D$.
\end{thm}

\begin{proof}
  First, we recall the Laplace estimate of $\omega_\beta$, i.e. there exists $C_1$, $C_2$ such that
  \begin{equation*}
    C_1\omega_0\leq \omega_\beta\leq \frac{C_2}{|s|_H^{\frac{2-2\beta}{\lambda}}}\omega_0.
  \end{equation*}
  It is easy to see that $\omega_\beta$ defines a metric on $M$, and makes it into a compact length space. This can also be considered as the metric completion of the incomplete Riemannian manifold $(X\backslash D,\omega_\beta)$. On the other hand, it is not immediately clear that $(X\backslash D,\omega_\beta)$ is geodesically convex.

  By Theorem \ref{Theorem:wz}, we denote $(Y,d_Y)$ to be a sequential Gromov--Hausdorff limit of $(M,\omega_{\phi_\varepsilon})$. By Lemma \ref{lemma:diameter}, we know that $(Y,d_Y)$ is a compact length space. Since $\omega_{\phi_\varepsilon}$ converges smoothly to $\omega_\beta$ locally on $M\backslash D$, we obtain a smooth dense open subset $U$ of $Y$ endowed with a Riemannian metric $g_\infty$, and a surjective local isometry
  \begin{equation*}
  F_\infty:(X\backslash D,\omega_\beta)\to (U,g_\infty)\text{ (as Riemannian manifolds)}.
  \end{equation*}
  For any $x,y\in U$, clearly we have $d_Y(x,y)\leq d_U(x,y)$, where $d_U$ is the metric on $U$ induced by the Riemannian metric $g_\infty$. From this it is easy to see that $F_\infty$ is a Lipschitz map, so $F_\infty$ extends to a Lipschitz map from the metric completion $(X,\omega_\beta)$ to $Y$, and the image is closed. Since $U$ is dense, it follows that $F_\infty$ is surjective. It is clear that $Y\backslash U$ is contained in $F_\infty(D)$. Since $D$ has zero codimension one Hausdorff measure with respect to $\omega_\beta$, we see that $Y\backslash U$ also has zero codimension one Hausdorff measure. Then by Theorem 3.7 of \cite{cheeger2000structure}, we know that $d_U(x,y)=d_Y(x,y)$ for any $x,y\in U$, and $Y$ is the metric completion of $(U,d_U)$. It follows that $F_\infty$ is an isometry.
\end{proof}



\begin{thebibliography}{10}

\bibitem{aubin1976equations}
Thierry Aubin.
\newblock \textit{Equations du type {M}onge-{A}mp\'{e}re sur les varietes
  {K}\"{a}hleriennes compactes.}
\newblock {CR Acad. Sci Paris.}, 283:119--121, 1976.

\bibitem{berman2013thermodynamical}
Robert~J Berman.
\newblock \textit{A thermodynamical formalism for {M}onge-{A}mpere equations,
  {M}oser-{T}rudinger inequalities and {K}{\"a}hler--{E}instein metrics.}
\newblock {Advances in Mathematics}, 248:1254--1297, 2013.

\bibitem{berndtsson2013brunn}
Bo~Berndtsson.
\newblock \textit{A {B}runn-{M}inkowski type inequality for {F}ano manifolds and some
  uniqueness theorems in {K}\"{a}hler geometry.}
\newblock {Inventiones mathematicae};1--52 2013.

\bibitem{cao2012existence}
Huai-Dong Cao.
\newblock \textit{Existence of gradient {K}\"ahler-{R}icci solitons.}
\newblock {arXiv:1203.4794}, 2012.

\bibitem{cao2005kahler}
Huai-Dong Cao, Gang Tian, and Xiaohua Zhu.
\newblock \textit{K{\"a}hler--{R}icci solitons on compact complex manifolds with
  $c_1\lr{M}>0$.}
\newblock {Geometric \& Functional Analysis GAFA}, 15(3):697--719, 2005.

\bibitem{cheeger2000structure}
Jeff Cheeger, Tobias~H Colding, et~al.
\newblock \textit{On the structure of spaces with {R}icci curvature bounded below.
  {II}.}
\newblock {Journal of Differential Geometry}, 54(1):13--35, 2000.

\bibitem{chen2012kahler}
Xiuxiong Chen, Simon Donaldson, and Song Sun.
\newblock \textit{K{\"a}hler-einstein metrics on {F}ano manifolds. {I}: Approximation
  of metrics with cone singularities.}
\newblock {Journal of the American Mathematical Society}, 2012.

\bibitem{datar2013connecting}
Ved Datar, Bin Guo, Jian Song, and Xiaowei Wang.
\newblock \textit{Connecting toric manifolds by conical {K}\"ahler-{E}instein metrics.}
\newblock {arXiv:1308.6781}, 2013.

\bibitem{donaldson2012cone}
Donaldson, SK.
\newblock \textit{K{\"a}hler metrics with cone singularities along a divisor.}
\newblock {Essays in Mathematics and its Applications, Springer}, 2012.

\bibitem{futaki1995bilinear}
Akito Futaki and Toshiki Mabuchi.
\newblock \textit{Bilinear forms and extremal {K}{\"a}hler vector fields associated
  with {K}{\"a}hler classes.}
\newblock {Mathematische Annalen}, 301(1):199--210, 1995.

\bibitem{guenancia2013conic}
Henri Guenancia and Mihai P{\u{a}}un.
\newblock \textit{Conic singularities metrics with prescribed {R}icci curvature: the
  case of general cone angles along normal crossing divisors.}
\newblock {arXiv:1307.6375}, 2013.

\bibitem{GS} G. Szekelyhidi, Greatest lower bounds on the Ricci curvature of Fano manifolds. Compos. Math.
147(2011), no.1, 319¨C331.

\bibitem{hamilton1993eternal}
Richard~S Hamilton et~al.
\newblock \textit{Eternal solutions to the {R}icci flow.}
\newblock {Journal of Differential Geometry}, 38(1):1--11, 1993.

\bibitem{kolodziej2008holder}
S{\l}awomir Ko{\l}odziej.
\newblock \textit{H{\"o}lder continuity of solutions to the complex
  {M}onge-{A}mp{\`e}re equation with the right-hand side in $ {L}^p$: the case
  of compact {K}{\"a}hler manifolds.}
\newblock {Mathematische Annalen}, 342(2):379--386, 2008.

\bibitem{li2013yau}
Chi Li.
\newblock \textit{Yau-{T}ian-{D}onaldson correspondence for {K}-semistable {F}ano
  manifolds.}
\newblock {arXiv:1302.6681}, 2013.

\bibitem{li2012conical}
Chi Li and Song Sun.
\newblock \textit{Conical {K}ahler-{E}instein metric revisited.}
\newblock {Communications in Mathematical Physics}, 331(3):927--973, 2014.

\bibitem{liu2014conical}
Jiawei Liu and Xi~Zhang.
\newblock \textit{The conical {K}\"ahler-{R}icci flow on {F}ano manifolds.}
\newblock {arXiv:1402.1832}, 2014.

\bibitem{mabuchi2003multiplier}
Toshiki Mabuchi et~al.
\newblock \textit{Multiplier {H}ermitian structures on {K}{\"a}hler manifolds.}
\newblock {Nagoya Mathematical Journal}, 170:73--115, 2003.

\bibitem{phong2008moser}
Duong~H Phong, Jian Song, Jacob Sturm, and Ben Weinkove.
\newblock \textit{The {M}oser-{T}rudinger inequality on {K}{\"a}hler-{E}instein
  manifolds.}
\newblock {American Journal of Mathematics}, 130(4):1067--1085, 2008.

\bibitem{simpson1988constructing}
Simpson, Carlos T.
\newblock\textit{ Constructing variations of Hodge structure using Yang-Mills theory and applications to uniformization.}
\newblock { Journal of the American Mathematical Society}, 1(4):867--918, 1998.

\bibitem{song2012greatest}
Jian Song and Xiaowei Wang.
\newblock \textit{The greatest {R}icci lower bound, conical {E}instein metrics and the
  {C}hern number inequality.}
\newblock { arXiv:1207.4839}, 2012.

\bibitem{tian1997kahler}
Gang Tian.
\newblock \textit{K{\"a}hler-{E}instein metrics with positive scalar curvature.}
\newblock { Inventiones Mathematicae}, 130(1):1--37, 1997.

\bibitem{tian2000book}
Gang Tian.
\newblock \textit{ Canonical metrics in {K}{\"a}hler geometry.}
\newblock Springer, 2000.

\bibitem{tian2013k}
Gang Tian.
\newblock \textit{K-stability and {K}\"ahler-{E}instein metrics.}
\newblock { arXiv:1211.4669 v2}, 2012.

\bibitem{tian2000nonlinear}
Gang Tian and Xiaohua Zhu.
\newblock \textit{A nonlinear inequality of {M}oser-{T}rudinger type.}
\newblock { Calculus of Variations and Partial Differential Equations},
  10(4):349--354, 2000.

\bibitem{tian2000uniqueness}
Gang Tian and Xiaohua Zhu.
\newblock \textit{Uniqueness of {K}{\"a}hler-{R}icci solitons.}
\newblock { Acta Mathematica}, 184(2):271--305, 2000.

\bibitem{tian2002new}
Gang Tian and Xiaohua Zhu.
\newblock \textit{A new holomorphic invariant and uniqueness of {K}{\"a}hler-{R}icci
  solitons.}
\newblock { Commentarii Mathematici Helvetici}, 77(2):297--325, 2002.

\bibitem{wang2013structure}
Feng Wang and Xiaohua Zhu.
\newblock \textit{On the structure of spaces with {B}akry-{E}mery {R}icci curvature
  bounded below.}
\newblock { arXiv:1304.4490}, 2013.

\bibitem{wang2014toric}
Wang, Feng and Zhou, Bin and Zhu, Xiaohua.
\newblock \textit{Modified Futaki invariant and equivariant Riemann-Roch formula.}
\newblock { arXiv:1408.3784}, 2014.

\bibitem{wei2007comparison}
Guofang Wei and William Wylie.
\newblock \textit{ Comparison geometry for the bakry-emery ricci tensor.}
\newblock { Journal of differential geometry}, 83(2):377--406, 2009.

\bibitem{yau1978ricci}
Shing-Tung Yau.
\newblock \textit{On the ricci curvature of a compact {K}{\"a}hler manifold and the
  complex {M}onge-{A}mp{\'e}re equation, }{I}.
\newblock {Communications on pure and applied mathematics}, 31(3):339--411,
  1978.

\bibitem{zhanggeneralized1}
Xi~Zhang and Xiangwen Zhang.
\newblock \textit{generalized {K}\"ahler-{E}instein metrics and energy functionals.}
\newblock { Canadian Journal of Mathematics}, 2013.

\bibitem{zhu2000kahler}
Xiaohua Zhu.
\newblock\textit{ K{\"a}hler-{R}icci soliton typed equations on compact complex
  manifolds with $c_1\lr{M}>0$.}
\newblock { The Journal of Geometric Analysis}, 10(4):759--774, 2000.




\end{thebibliography}
\bibliographystyle{amsplain}

\end{document}